\documentclass[a4paper,oneside,10pt]{article}

\usepackage[english]{babel}
\usepackage[latin1]{inputenc}
\usepackage{lmodern,color}
\usepackage[T1]{fontenc}
\usepackage{indentfirst}
\usepackage{amsmath,amssymb}

\usepackage[a4paper,top=3cm,bottom=3cm,left=3cm,right=3cm,%
bindingoffset=5mm,marginparsep=0cm,hoffset=-0.5cm,marginparwidth=0cm]{geometry}
\usepackage{lmodern}
\usepackage[T1]{fontenc}
\usepackage{lettrine}
\usepackage{booktabs}

\usepackage{authblk}
\usepackage{emp}
\usepackage{amsthm}
\usepackage{amsmath,amssymb,amsthm}
\usepackage{braket}
\usepackage{chngpage}
\usepackage{cases}
\usepackage{graphicx}
\usepackage{epstopdf}
\usepackage{tabularx}
\usepackage{multirow}

\usepackage[textwidth=0.2\textwidth]{todonotes}

\renewcommand{\epsilon}{\varepsilon}
\renewcommand{\theta}{\vartheta}
\renewcommand{\rho}{\varrho}
\renewcommand{\phi}{\varphi}
\newcommand{\nub}{\boldsymbol{\nu}}
\newcommand{\parab}{\boldsymbol{\rho}}
\newcommand{\para}{\rho}
\newcommand{\gammab}{\boldsymbol{\gamma}}
\newcommand{\epsb}{\boldsymbol{\varepsilon}}
\newcommand{\sigmab}{\boldsymbol{\sigma}}
\newcommand{\taub}{\boldsymbol{\tau}}

\newcommand{\ub}{\mathbf{u}}
\newcommand{\vb}{\mathbf{v}}

\newcommand{\ab}{\mathbf{a}}
\newcommand{\bb}{\mathbf{b}}
\newcommand{\xb}{\mathbf{x}}
\newcommand{\pb}{\mathbf{p}}

\newcommand{\Pp}{{\mathcal P}}

\newcommand{\Amat}{\mathbb{A}}
\newcommand{\Bmat}{\mathbb{B}}

\newcommand{\veOne}{\texttt{ve1}}
\newcommand{\veTwo}{\texttt{ve2}}

\def\P0{{\Pi^{0, E}_k}}

\theoremstyle{definition}

\theoremstyle{remark}
\newtheorem{remark}{Remark}[section]
\theoremstyle{remark}

\theoremstyle{plain}

\newtheorem{proposition}{Proposition}[section]

\newtheorem{lemma}{Lemma}[section]

\usepackage{authblk}

\usepackage{setspace}

\author[1]{E. Artioli \thanks{artioli@ing.uniroma2.it}}
\author[2]{L. Beir\~ao da Veiga \thanks{lourenco.beirao@unimib.it ({\bf corresponding author})}}
\author[2]{F. Dassi \thanks{franco.dassi@unimib.it}}

\affil[1]{Department of Civil Engeneering and Computer science,

University of Rome Tor Vergata, Via del Politecnico 1 - 00133 Roma, Italy} 

\affil[2]{Department of Mathematics and Applications, 

University Milano - Bicocca, Via R. Cozzi 55 - 20125 Milano, Italy}  

\title{
\textbf{Curvilinear Virtual Elements for 2D solid mechanics applications}}

\begin{document}

\maketitle


\begin{abstract}
In the present work we generalize the curvilinear Virtual Element technology, introduced for a simple linear scalar problem in a previous work, to generic 2D solid mechanic problems  in small deformations. Such generalization also includes the development of a novel Virtual Element space for displacements that contains rigid body motions.
Our approach can accept a generic black-box (elastic or inelastic) constitutive algorithm and, in addition, can make use of curved edges thus leading to an exact approximation of the geometry. 
Rigorous theoretical interpolation properties for the new space on curvilinear elements are derived.
We undergo an extensive numerical test campaign, both on elastic and inelastic problems, to assess the behavior of the scheme. The results are very promising and underline the advantages of the curved VEM approach over the standard one in the presence of curved geometries.
\end{abstract}

\section{Introduction}

The Virtual Element Method (VEM) is a recent technology introduced in \cite{volley,autostoppisti} (see also \cite{apollo}) for the discretization of problems in partial differential equations, generalizing the Finite Element Method to meshes using general polygonal/polyhedral elements. In the framework of Structural Mechanics, after the initial papers 
\cite{BBM-elast,gain2014virtual,BLM,wriggers2016virtual}, many applications and developments were pursued by the community. Among the many publications, we here refer to a long but non exhaustive list 
\cite{Brezzi:Marini:plates,gain2015topology,Paulino2015bridging,PartI,PartII,Chi2017,Ant2017,wriggers2017low,Wriggers-2,artioli2018asymptotic,ArtMarSacc2018,ARTIOLI2018978,Ben2018,TA2018,PinTrov2019,DassiArxivHR2109,CHI201921}. Finally, it must be noted that other polytopal methods exist in the Structural Mechanics literature; we here limit ourselves in citing 
\cite{sukumar2004conforming,biabanaki2014polygonal,chi2015polygonal,di2015hybrid} 
and references therein. 

Since standard Virtual and Finite Elements make use of meshes with straight edges, the domain of interest is approximated by a linear interpolation of the boundary, a procedure which introduces a geometry approximation error that can dominate the analysis for second-to-higher order schemes. In the literature there has been introduced many approaches and variants in order to have a better (or even exact) approximation of the geometry. Among the many possibilities, we here cite two among the most known, that are isoparametric Finite Elements (see for instance \cite{libro_sui_FEM_di_Ciarlet}) and Isogeometric Analysis (see for instance 
\cite{hughes2005isogeometric,cottrell2009isogeometric}).
As shown in \cite{BRV_curvi}, it turns out that for Virtual Elements in 2D the extension to cases with curved edges (and thus able to approximate exactly the geometry of interest) is quite natural. By changing the definition of the edge-wise space, the method is able to accommodate for general boundaries described by a given parametrization (for instance obtained by CAD). The approximation properties of the space and the convergence of the method turns out to be optimal, and there is no need to build a volumetric parametrization of the object as in standard IGA or to carefully position the ``mapping points'' as in isoparametric FEM. Currently, a limitation of this VEM approach is that extension to 3D is not immediate and needs more investigation.
{ Furthermore, we cite also \cite{AnanArxiv,BertoArxiv,Botti2018} that represent other examples of approaches that allow for curved polygonal elements, still on a simple linear elliptic model problem as in \cite{BRV_curvi}.}

In the present work we start from the basic construction introduced in \cite{BRV_curvi}, that was on a simple (scalar) linear model problem, and generalize it to the more complex situation of small deformation problems in structural mechanics. Our approach, in the spirit of \cite{BLM,PartI,PartII}, can accept a generic black-box (elastic or inelastic) constitutive algorithm but, in addition, can make use of curved edges thus leading to an exact approximation of the geometry. When the Virtual Spaces of \cite{BRV_curvi} are applied in the present context of structural mechanics, a potential drawback appears. Indeed, such spaces (although holding optimal approximation properties) contain constant functions but do not contain linear ones. As a consequence, the ensuing vector displacement space does not contain rigid body rotations. In order to deal with this issue, in the present paper we propose also an alternative novel VEM space, that has the same degrees of freedom as the original one but the added property of containing all rigid body motions. Such space is based on a non-conventional construction that is partially in physical space and partially mapped from a parameter space. We furthermore prove the optimal approximation properties of this new space, the proof being interesting due to the particular nature of the proposed space. 

After presenting the Virtual space, the associated method and the theoretical results, we undergo an extensive numerical test campaign, both on elastic and inelastic problems, to assess the behavior of the scheme and make some comparison. The results are very promising and underline the advantages of the curved VEM approach over the standard one for elements of second (or higher) order.
It is notable that, from the pure coding standpoint, the extension from standard to curved elements is essentially only a change in the quadrature rules. This makes the proposed approach a viable choice, since modifying a standard VEM scheme into a curved approach implies a modification (and not a revolution) in the existing codes. 

The paper is organized as follows. After introducing some preliminaries in Sec.~\ref{sec:intro}, we describe the two VEM spaces (original and novel) in Sec.~\ref{sec:1}. Afterwards, we derive the theoretical interpolation properties of the new space in Sec.~\ref{sec:2}, and describe the proposed method in Sec.~\ref{sec:3}. Finally, Sec.~\ref{sec:num} is devoted to the numerical tests, divided into elastic and inelastic cases.

\section{Notation and preliminaries}\label{sec:intro}

We assume a bounded 2D domain $\Omega$ describing the body of interest, such that its boundary $\partial\Omega$ is a finite union of regular components $\Gamma_i$, $i=1,2,..,m$. Each of such regular components is parametrized by an invertible mapping $\gamma_i$ from an interval in the real line into $\Gamma_i$ (this is a typical situation, for instance, for domains described by CAD). We assume that we have a conforming mesh of polygons $\Omega_h$, that may have curved edges on the boundary in order to describe exactly the domain $\Omega$ (see Fig.~\ref{Fig-Geo1}). It is not restrictive to assume that each curved face is a subset of only one $\Gamma_i$ and therefore regular. In order to symplify the notation in the following we will simply use $\Gamma$ and
$$
\gammab \ : \ [0,L] \longrightarrow \Gamma .
$$
to indicate a generic curved part of the boundary, and its associated parametrization.

\begin{figure}
\begin{center}
\includegraphics[scale=0.60]{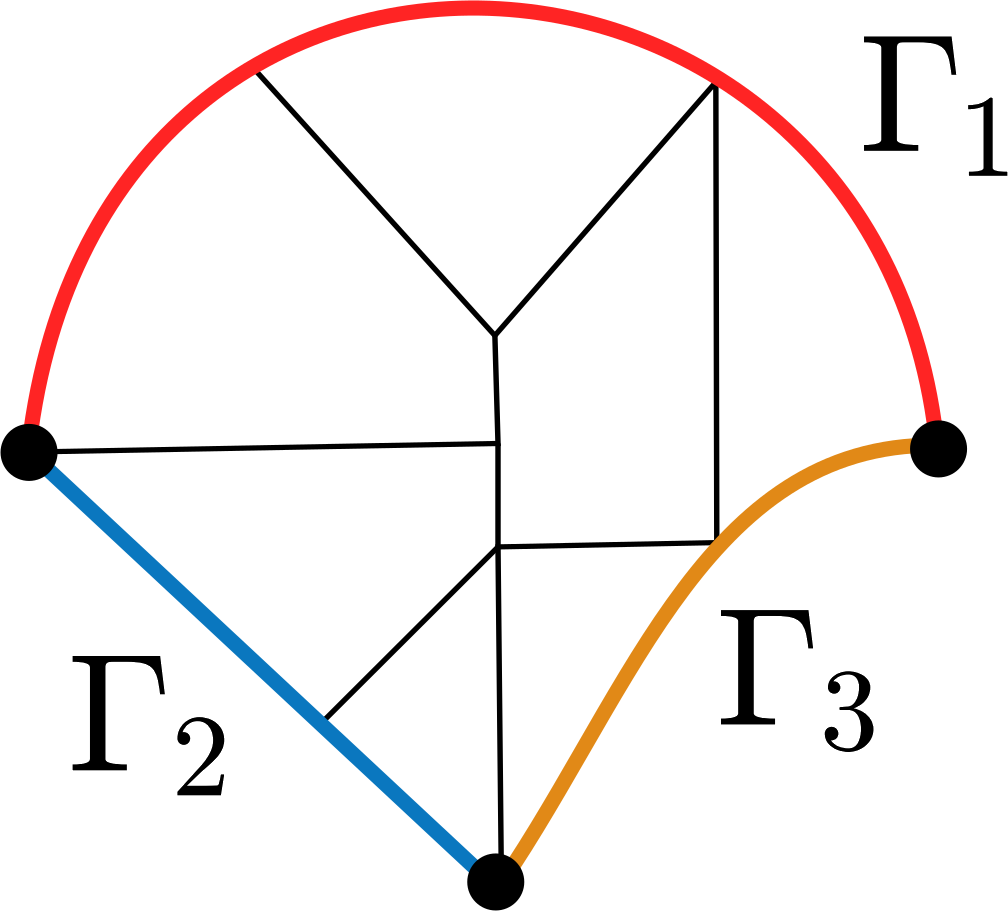}
\end{center}
\caption{Sample meshed domain $\Omega$.}
\label{Fig-Geo1}
\end{figure}

In the following we will denote with $E$ a generic element of the mesh and with $e$ a generic edge. As usual the symbol $h$ will be associated to the diameter of objects, for instance $h_E$ will denote the diameter of the element $E$ and $h_e$ the (curvilinear) length of the edge $e$. An $h$ without indexes denotes as usual the maximum mesh element size.

We will adopt the standard notation for $H^s$ Sobolev norms and semi-norms (with $s \in \mathbb R$ non-negative) on open measurable sets. In the case of a curved edge, the norm is intended with respect to the arc-length parametrization of the edge. We will denote by $\lesssim, \gtrsim$ an inequality of two real numbers that holds up to a costant that is independent of the particular mesh $\Omega_h$ or, in case of a local estimate, of the element $E$.

In this paper we focus on the equilibrium problem for a two-dimensional medium $\Omega$ with either elastic or inelastic constitutive material behavior, in small deformation regime. More specifically, in the elastic case the variational formulation of the problem reads
\begin{equation}\label{eq:con-pbl}
\left\{
\begin{aligned}
& \textrm{Find } \ub \in {\cal V} \textrm{ such that} \\
& \int_\Omega \sigmab({\bf x},\epsb(\ub)): \epsb(\vb) {\rm d}x = F(\vb) \qquad \forall \vb \in \delta {\cal V} ,
\end{aligned}
\right.
\end{equation}
where ${\cal V}$ (and $\delta {\cal V}$) represents the space of kinematically admissible displacements (and its variations), 
$\epsb$ denotes the usual linear strain operator, $\sigmab$ represents the elastic constitutive law relating strains to stresses, $F$ indicates a linear operator representing volume and surface forces acting on $\Omega$.

In the case of an inelastic material behaviour, the constitutive law $\sigma$ will depend on some history variables and the problem will be written in a pseudo-time evolution formulation, see for instance \cite{simo_computational_1998}.
\begin{remark}
We note that the method here proposed trivially extends also to the case where there are (fixed) interfaces also inside $\Omega$, for instance in the presence of material coefficient jumps. In such case also internal elements may have curved edges in order to adapt to the material interface. For the sake of simplicity, in the exposition we will refer only to the boundary of $\Omega$ as a source of curved edges. 
\end{remark}

\section{Description of the Virtual Space}
\label{sec:1}

In the present section we briefly review the space proposed in \cite{BRV_curvi} and introduce also an alternative space that has the additional advantage of including rigid body motions. Indeed, the space proposed for the diffusion problem in \cite{BRV_curvi} has good approximation properties and is therefore suitable also as a displacement space for problems in structural mechanics. On the other hand, such space does contain constant polynomials (that is translations of the body) but not (linearized) rigid body rotations. Therefore we here propose an alternative space that includes also rigid body rotations without increasing the number of degrees of freedom.

\subsection{The original space}
\label{sec:1-1}

We here quickly review the space of \cite{BRV_curvi}, here extended to the vector-valued version.
As usual, we define the space element by element. Let therefore $E \in \Omega_h$. Note that $E$ may have some curved edge, laying on the curved boundary $\Gamma$. For each curved edge $e$, let $\gammab_e : [a,b] \rightarrow e$ denote the restriction of the parametrization describing $\Gamma$ to the edge $e$. 
Then, we indicate the space of mapped polynomials (living on $e$) as 
$$
\widetilde{\Pp}_k(e) = \Big\{ \pb \circ \gammab_e^{-1} \: : \: \pb \in (\Pp_k[a,b])^2 \Big\} . 
$$
The local virtual element space on $E$ is then defined as
\begin{equation}\label{original}
\begin{aligned}
V_h(E) = \Big\{ \vb \in [H^1(E) \cap C^0(E)]^2 \: & : \: \vb|_e \in [\Pp_k(e)]^2 \textrm{ if } e \textrm{ is straight} , \\ 
& \vb|_e \in [\widetilde{\Pp}_k(e)]^2 \textrm{ if } e \textrm{ is curved} , 
- \Delta \vb \in [\Pp_{k-2}(E)]^2 \Big\} .
\end{aligned}
\end{equation} 
The associated degrees of freedom are (see \cite{BRV_curvi} for the simple proof)
\begin{itemize}
\item\textbf{D1}: pointwise evaluation at each vertex of $E$ ,
\item\textbf{D2}: pointwise evaluation at $k-1$ distinct points for each edge of $E$
\item\textbf{D3}: moments $\int_E \vb \cdot \pb_{k-2}$ for all $\pb_{k-2} \in [\Pp_{k-2}(E)]^2$ .
\end{itemize}
As usual, the global space is obtained by a standard gluing procedure
$$
V_h = \Big\{ \vb \in H^1(\Omega) \: : \: \vb|_E \in V_h(E) \ \forall E \in \Omega_h \Big\}
$$
with the obvious extension of the local degrees of freedom to global ones.
Note that, clearly, Remark \ref{rem:enh} below applies also to the space above.

\subsection{An alternative space}
\label{sec:1-2}

Let $e$ be a (possibly curved) edge of the polygon $E$, with endpoints $\overline{\nub},\nub'$. In the following we denote by $h_e$ the (curvilinear) length of the edge $e$. Note that, as a consequence of the fact that the boundary $\Gamma$ is fixed once and for all (see \cite{BRV_curvi}) it holds $h_e \simeq \| \overline{\nub} - \nub'\|$.

We will now present a preliminary linear space of (two component) vector fields living on $e$, with associated degrees of freedom given by the vector evaluation at the two endpoints of $e$. Let therefore $\overline{\ub}$ and $\ub'$ in $\mathbb{R}^2$ represent vector values associated to $\overline{\nub}$ and $\nub'$, respectively. We build an associated vector field on $e$ as follows. Let $F : \mathbb{R}^2 \rightarrow \mathbb{R}^2$ be the unique mapping that satisfies
\begin{equation}\label{F-def}
\left\{
\begin{aligned}
& F (\xb) = \ab + \Amat \, \bb \quad \ab, \bb \in \mathbb{R}^2 , \\
& F(\overline{\nub}) = \overline{\ub} \: , \ F(\nub') = \ub' \ , 
\end{aligned}
\right. 
\end{equation}
where the matrix field $\Amat=\Amat(\xb)$ depends on the position $\xb=(x_1,x_2)$
$$
\Amat = 
\begin{pmatrix}
x_1 - \overline{\nu}_1 & - (x_2 - \overline{\nu}_2) \\
x_2 - \overline{\nu}_2 & x_1 - \overline{\nu}_1 
\end{pmatrix} .
$$ 
A mapping $F$ as above is a composition of a translation, an homotopy and a (linearized) rotation; therefore mappings of this type include, in particular, linearized rigid body motions (that is the kernel of the symmetric gradient operator).  Since the map $F$ depends on ${\overline{\ub},\ub'}$ we will denote it also by $F^{\overline{\ub},\ub'} $ when we need to be explicit on such dependence. We will give an explicit expression of $F$ in Lemma 
\ref{lem:F-expl} below.

Then, the space $R_h(e)$ of two-component vector fields living on $e$ is defined as the restriction to $e$ of all possible maps $F$:
$$
R_h(e) = \Big\{ F^{\overline{\ub},\ub'}|_{e} \: : \: \overline{\ub},\ub' \in \mathbb{R}^2 \Big\} . 
$$
The following lemma follows by direct computations (the simple proof is not shown).
\begin{lemma}\label{lem:F-expl}
Given $\overline{\ub},\ub' \in \mathbb{R}^2$, the coefficients ${\bf a}, {\bf b}$ of the map $F = F^{\overline{\ub},\ub'}$ in \eqref{F-def} are
$$
\begin{aligned}
& {\bf a} = \overline{\ub} \ , \quad {\bf b} =  {\Bmat}^{-1} (\ub' - \overline{\ub}) \ , \\
& \Bmat = 
\begin{pmatrix}
\nu'_1 - \overline{\nu}_1 & - (\nu'_2 - \overline{\nu}_2) \\
\nu'_2 - \overline{\nu}_2 & \nu'_1 - \overline{\nu}_1 
\end{pmatrix} \ , \quad
\Bmat^{-1} = \frac{1}{\| \nub' - \overline{\nub} \|^2}
\begin{pmatrix}
\nu'_1 - \overline{\nu}_1 & \nu'_2 - \overline{\nu}_2 \\
-(\nu'_2 - \overline{\nu}_2) & \nu'_1 - \overline{\nu}_1 
\end{pmatrix} .
\end{aligned}
$$
\end{lemma}
We recall $\gammab_e : [a,b] \rightarrow e$ denotes the restriction of the parametrization describing 
$\Gamma$ to the edge $e$. We introduce another space of vector fields living on $e$, defined as the push forward of the $\Pp_k$ polynomials living on $e$ that vanish on $a,b$:
$$
B_h(e) = [\widetilde{\Pp}_k(e) \cap H^1_0(e)]^2 =
\Big\{ \pb \circ \gammab_e^{-1} \: : \: \pb \in (\Pp_k[a,b])^2, \: \pb(a)=\pb(b)={\bf 0} \Big\} . 
$$
Finally, the virtual edge space is
\begin{equation}\label{qqq}
V_h(e) = R_h(e) \oplus B_h(e) .
\end{equation}
{
\begin{remark}\label{rem:dofedge}
A set of degrees of freedom for the space $V_h|e$ is simply given by $k+1$ distinct pointwise evaluations along the edge, including the two extrema. This is easy to check, exploiting the fact that the evaluation operators at the two endpoints constitutes a set of degrees of freedom for $R_h(e)$ that vanish when computed on $B_h(e)$.
\end{remark}
As a byproduct of the above remark, the two spaces in \eqref{qqq} constitute a direct sum and the dimension of the space $V_h(e)$ is equal to $k+1$. 
}
We are now ready to define the virtual space $V_h(E)$ 
on the (possibly curved) polygon $E$.
\begin{equation}\label{new}
\begin{aligned}
V_h(E) = \Big\{ \vb \in [H^1(E) \cap C^0(E)]^2 \: & : \: \vb|_e \in [\Pp_k(e)]^2 \textrm{ if } e \textrm{ is straight} , \\ 
& \vb|_e \in V_h(e) \textrm{ if } e \textrm{ is curved} , - \Delta \vb \in [\Pp_{k-2}(E)]^2 \Big\} .
\end{aligned}
\end{equation}
Note that, due to Remark \ref{rem1} below, the definition above is robust to the limit when a curved edge becomes straight. The standard virtual element theory \cite{volley}, combined with Remark \ref{rem:dofedge}, yields that 
\textbf{D1},~\textbf{D2} and \textbf{D3} of Sec.~\ref{sec:1-1} is also a set of degrees of freedom for the novel space $V_h(E)$.
As usual, the global space is obtained by a standard gluing procedure
$$
V_h = \Big\{ \vb \in H^1(\Omega) \: : \: \vb|_E \in V_h(E) \ \forall E \in \Omega_h \Big\}
$$
with the obvious extension of the local degrees of freedom to global ones.

\begin{remark}\label{rem1}
It can be checked that, whenever the edge e is straight and the associated parametrization $\gamma_e$ is affine, one has $V_h(e) = \Pp_k(e)$. Therefore in such case we correctly recover the standard virtual space on straight edges.
\end{remark}

\begin{remark}\label{rem:enh}
Also an enhanced version (see for instance \cite{projectors}) of this space could be introduced. Indeed, since the enhancement of the space does not involve the boundary of the elements, it is trivial to build the enhanced version of the spaces $V_h(E)$ above combining the present results with those in \cite{projectors,genCoeff,autostoppisti}.
\end{remark}

\section{Approximation properties of the space}
\label{sec:2}

Optimal approximation properties for the original space in the $H^1$ norm were proved in Theorem 3.1 of \cite{BRV_curvi}. These immediately extend (component-wise) to the vector valued version \eqref{original} here considered. On the other hand, the space \eqref{new} differs from \eqref{original}.
The only difference between the space \eqref{original} and \eqref{new} is the definition of the boundary space on curved edges (that is $V_h(e)$ instead of 
$[\widetilde{\Pp}_k(e)]^2$). Therefore we can directly apply the theory of \cite{BRV_curvi} also in this new case, provided we can prove approximation estimates for the new curved edge space $V_h(e)$. That is, we need to prove a new counterpart of Lemma 3.2 in \cite{BRV_curvi} for the novel space $V_h(e)$, the challenge being that $V_h(e)$ is \emph{not} a mapped polynomial space but stems instead from a ``hybrid'' definition \eqref{qqq}. This is the topic of the present section.

Let now a generic vector field $\ub \in [H^1(e)]^2$.  
Let the unique interpolant
\begin{equation}\label{eq:def1}
\ub_I^R \in R_h(e) \ , \quad 
\ub_I^R(\overline{\nub}) = \ub(\overline{\nub}) , \
\ub_I^R(\nub') = \ub(\nub') .
\end{equation}
\begin{lemma}\label{lem:bound}
It exists a positive constant C, depending only on $k$ and the parametrization $\gammab$, such that
$$
| \ub_I^R |_{H^m(e)} \le C  | \ub |_{H^1(e)} \qquad \forall \ m \in \{ 1,2,..,k+1 \}.
$$
\end{lemma}
\begin{proof}
By definition $\ub_I^R = F|e$, where $F$ is given in \eqref{F-def} and coefficients determined by Lemma \ref{lem:F-expl} with $\overline{\ub} = \ub(\overline{\nub})$ and $\ub'=\ub(\nub')$. Let $t$ represent the arc-length parametrization along the edge $e$, and 
$\parab : [0,h_e] \rightarrow e$ the corresponding parametrization. 
Since all norms on $\mathbb{R}^{2 \times 2}$ are equivalent, it holds
$$
\| \Bmat^{-1} \| \lesssim \left(\Bmat^{-T}:\Bmat^{-1}\right)^{1/2} = \| \nub' - \overline{\nub} \|^{-1} \lesssim h_e^{-1} .
$$
Then, by an Holder inequality, the euclidean norm of ${\bf b}$ satisfies
\begin{equation}\label{pippo}
\begin{aligned}
\| {\bf b} \| & \le \| \Bmat^{-1} \| \: \| \ub' - \overline{\ub} \| = 
\| \Bmat^{-1} \| \: \| \ub(\nub') - \ub(\overline{\nub}) \| \\
& \lesssim  h_e^{-1} \| \!\! \int_{\overline{\nub}}^{\nub'} \partial_t \ub(t) {\rm d}t \|
\lesssim  h_e^{-1/2} | \ub |_{H^1(e)} .
\end{aligned}
\end{equation}
The expression of $\ub_I^R$ in term of the arc-length parametrization is, by definition,
$$
\ub_I^R(t) = \ab + \Amat \, \bb \ \textrm{with} \quad
\Amat = 
\begin{pmatrix}
\para_1(t) - \overline{\nu}_1 & - (\para_2(t) - \overline{\nu}_2) \\
\para_2(t) - \overline{\nu}_2 & \para_1(t) - \overline{\nu}_1 
\end{pmatrix} .
$$
Therefore a direct derivation and \eqref{pippo}
yield, for all $t \in (0,h_e)$ and $m \in \{ 1,2,..,k+1 \}$,
\begin{equation}\label{eq1}
\| \partial^m_t \ub_I^R(t) \| = \| 
\begin{pmatrix}
\partial^m_t \para_1(t) & - \partial^m_t \para_2(t)  \\
\partial^m_t \para_2(t) & \partial^m_t \para_1(t) 
\end{pmatrix} {\bf b} \| 
\lesssim \| \partial^m_t \parab(t) \| \| {\bf b} \| 
\lesssim h_e^{-1/2} | \ub |_{H^1(e)} .
\end{equation}
Note that the term above $\| \partial^m_t \parab(t) \| \lesssim 1$,  independently of the mesh, because the curve $\Gamma$ is fixed once and for all. The Holder inequality together with bound \eqref{eq1} immediately grants
$$
| \ub_I^R(t) |_{H^m(e)} = \| \partial^m_t \ub_I^R(t) \|_{L^2(e)}
\lesssim h_e^{1/2} h_e^{-1/2} | \ub |_{H^1(e)} = | \ub |_{H^1(e)} .
$$
\end{proof}

We now need to introduce some notation. Given a curved edge $e$, we denote by $\{ x_i^e \}_{i=1}^{k+1}$ the $k+1$ distinct points on the edge associated to the DoFs of $V_h(e) = V_{h|e}$ (the cases $i=1,i=k+1$ representing the two extrema of $e$). 
Given any sufficiently regular vector field $\vb$ living on $e$, we denote by $\vb_{I\!I}$ its interpolant in $B_h(e)$, that is the unique function in $B_h(e)$ that satisfies $\vb_{I\!I}(x_i^e)=v(x_i^e)$ for $i=2,..,k$. We moreover denote by $\vb_I$ the interpolant of $\vb$ in the space $V_h(e)$, that is the unique function in $V_h(e)$ that satisfies $\vb_{I}(x_i^e)=\vb(x_i^e)$ for $i=1,..,k+1$. 
It is then easy to check that, for any sufficiently regular vector field $\ub$ living on $e$, one has 
\begin{equation}\label{eq2}
\ub_I = \ub_I^R + (\ub - \ub_I^R)_{I\!I} ,
\end{equation}
where we recall $u_I^R$ was defined in \eqref{eq:def1}.

We now prove our counterpart of Lemma 3.2 in \cite{BRV_curvi}. 
\begin{proposition}\label{Lem31counterpart}
Let $\ub \in [H^s(e)]$ with $1 \le s \le k+1$. Then it holds
$$
|\ub - \ub_I|_{H^m(e)} \lesssim h_E^{s-m} \|  \ub \|_{H^s(e)}
$$
for all real numbers $0 \le m \le s$.
\end{proposition}
\begin{proof}
Recalling \eqref{eq2} we have 
\begin{equation}\label{eqH0}
|\ub - \ub_I|_{H^m(e)} = |(\ub - \ub_I^R) - (\ub - \ub_I^R)_{I\!I}|_{H^m(e)} .
\end{equation}
Let now $\tilde\pb \in \widetilde{\Pp}_k(e)$ be the unique interpolant of $(\ub - \ub_I^R)$ at the points $\{ x_i^e \}_{i=1}^{k+1}$. It was proved in Lemma 3.2 of \cite{BRV_curvi} that the interpolation into mapped polynomials 
$\widetilde{\Pp}_k(e)$ enjoys standard approximation properties; therefore
\begin{equation}\label{eqH1}
|(\ub - \ub_I^R) - \tilde\pb|_{H^m(e)} \lesssim h_e^{s-m} \|  \ub - \ub_I^R \|_{H^s(e)} .
\end{equation}
Since the interpolation on mapped polynomials preserves constant functions (that is, the interpolation of a constant is the same constant), it is not restrictive to assume that $\ub - \ub_I^R$ has zero average. Therefore a standard Poincar\'e inequality for zero average functions yields that \eqref{eqH1} can be extended to the slightly sharper bound (where the $L^2$ part of the norm in the right hand side does not appear)
\begin{equation}\label{eqH2}
|(\ub - \ub_I^R) - \tilde\pb|_{H^m(e)} \lesssim h_e^{s-m} \sum_{j=1}^s |  \ub - \ub_I^R |_{H^j(e)} .
\end{equation}

We now note that, since $(\ub - \ub_I^R)$ vanishes at the endpoints of $e$, it holds 
$(\ub - \ub_I^R)_{I\!I} = \tilde\pb$. Therefore, combining this identity with \eqref{eqH0} and \eqref{eqH2}, we obtain
\begin{equation}\label{eq4}
|\ub - \ub_I|_{H^m(e)} \lesssim h_e^{s-m} \sum_{j=1}^s |  \ub - \ub_I^R |_{H^j(e)} .
\end{equation}
From \eqref{eq4}, we conclude by the triangle inequality and using Lemma \ref{lem:bound}
$$
|\ub - \ub_I|_{H^m(e)} \lesssim h_e^{s-m} \Big( \sum_{j=1}^m |  \ub |_{H^j(e)} +  |\ub_I^R |_{H^j(e)} \Big)
\lesssim h_e^{s-m} \sum_{j=1}^m |  \ub |_{H^j(e)} \lesssim h_E^{s-m} \|  \ub \|_{H^s(e)}.
$$
\end{proof}

Combining the results in \cite{BRV_curvi} with Proposition \ref{Lem31counterpart} (that substitutes Lemma 3.2 in \cite{BRV_curvi}) we immediately have the following approximation property for the global space $V_h$ associated to \eqref{new}.

\begin{proposition}\label{prop:interp}
Let $\ub \in H^s(\Omega)$ and $\frac{3}{2} < s \le k+1$. Let all elements $E \in \Omega_h$ be star shaped with respect to a ball of radius uniformly comparable to $h_E$.
Then, there exists $\ub_I \in V_h$ such that
$$
|\ub - \ub_I|_{H^1(\Omega)} \le C h^{s-1} \| \ub \|_{H^s(\Omega)} 
$$
where the constant $C$ depends on the degree $k$, the ``chunkiness'' constant associated to the shape regularity condition above and the parametrization $\gamma$.
\end{proposition}

The condition $s > \frac{3}{2}$ above could be generalized to $s > 1$, but we here prefer to follow the same choice as in \cite{BRV_curvi}. See Remark 3.8 in \cite{BRV_curvi}.

\section{Discretization of the problem}
\label{sec:3}

The discretization of the problem is a combination of the scheme proposed in \cite{BLM, PartI, PartII} for the case with standard straight edges and the curved-edge technology introduced in \cite{BRV_curvi} for a model linear diffusion problem. Clearly, we here have the possibility of using the (vector valued version) of the space in \cite{BRV_curvi} (briefly reviewed in Sec.~\ref{sec:1-1} here above) or the novel space proposed in Sec.~\ref{sec:1-2}. The construction below stays the same for both spaces.

We start by introducing the following projection operator that is used to compute, on each mesh element $E$, a polynomial strain function. Let $[\Pp_{k-1}(E)]^{2 \times 2}_{\rm sym}$ denote the set of polynomial symmetric tensors of degree $k-1$ living on $E$. Given $\vb_h \in V_h$ and $E \in \Omega_h$, the operator $\Pi^{\epsb} : V_h \rightarrow [\Pp_{k-1}(E)]^{2 \times 2}_{\rm sym}$ is defined by
$$
\left\{
\begin{aligned}
& \Pi^{\epsb}(v_h) \in [\Pp_{k-1}(E)]^{2 \times 2}_{\rm sym} \\
& \int_E \Pi^{\epsb}(v_h):\pb_{k-1} = \int_E \epsb(v_h):\pb_{k-1} \quad \forall \pb_{k-1} \in [\Pp_{k-1}(E)]^{2 \times 2}_{\rm sym} \: , \\
\end{aligned}
\right.
$$
where $\epsb(v_h)$ denotes as usual the symmetric gradient of $\vb_h$. On each element $E$, the polynomial tensor $\Pi^{\epsb}(v_h)$ represents the local approximation of the strains. Note that the above operator is computable. Indeed an integration by parts shows that
$$
\int_E \epsb(\vb_h):\pb_{k-1} = - \int_E \vb_h \cdot ({\rm div} \, \pb_{k-1})  \: + \: \int_{\partial E} v_h \cdot (\pb_{k-1}{\bf n}_E) \: .
$$
The first term on the right hand side can be computed noting that ${\rm div}\,\pb_{k-1}$ is a vector polynomial of degree $k-2$ and using the internal degrees of freedom values of $\vb_h$. The second term on the right hand side can be computed since we have complete knowledge of $\vb_h$ on the boundary of $E$. Note that all this computations clearly require the integration of known functions on a curved element and a curved boundary; those can be accomplished as shown for instance in 
\cite{BRV_curvi,vianello1,vianello2,vianello3}, see also Sec.~\ref{sec:num}.

We can now describe the proposed numerical method. For simplicity of exposition we will focus on the (nonlinear) elastic case \eqref{eq:con-pbl}, the inelastic case being treated analogously as commented in the notes below.
We start by introducing the local discrete``energy'' form (for any $E \in \Omega_h$ and $\ub_h,\vb_h \in V_h$)
divided into a \emph{consistency} and a \emph{stability} part
\begin{equation}\label{eq:form}
a_h^E(\ub_h,\vb_h) = \int_E \sigmab(\Pi^{\epsb}(\ub_h)): \Pi^{\epsb}(\vb_h) \: + \: s^E(\ub_h;(I-\Pi)\ub_h,(I-\Pi)\ub_h) \: ,
\end{equation}
where the above quantities and operators are described below.
The operator $\sigmab : {\mathbb R}^{2\times 2}_{\rm sym} \rightarrow {\mathbb R}^{2\times 2}_{\rm sym}$ represents the ``black box'' constitutive law associating strains to stresses (note that this rule may depend on the position ${\bf x} \in \Omega$).
The operator $\Pi:V_h(E) \rightarrow [\Pp_k(E)]^2$ can be chosen as any projection operator on vector-valued polynomials of degree $k$, for instance one that minimizes the distance of the euclidean norm of the degree of freedom values (such particular choice has the advantage of being very simple to code, see for instance \cite{PartI}). 
The stabilization form can be taken, for instance, as
\begin{equation}\label{eq:stab}
s^E(\ub_h;(I-\Pi)\ub_h,(I-\Pi)\ub_h) = \alpha(\ub_h) \sum_{i=1}^{\# dofs} 
\Big({\rm dof}_i (\ub_h-\Pi\ub_h) \Big) \Big({\rm dof}_i (\vb_h-\Pi\vb_h) \Big) 
\end{equation}
where the ${\rm dof}_i(\cdot)$ symbol denotes evaluation at the $i^{th}$ local degree of freedom. 
The positive scalar $\alpha$ depending on $\ub_h$ is important in order to scale the discrete form correctly with respect to the constitutive law \cite{BLM}. It can, for instance, be taken as
\begin{equation}\label{eq:alfa}
\alpha(\ub_h) = \left\| \frac{\partial \sigmab}{\partial \epsb} (\Pi\ub_h({\bf x}_b))  \right\|
\end{equation}
with $\| \cdot \|$ denoting any (chosen) norm on the tangent tensor, that is here computed for $\Pi\ub_h$ evaluated at the baricenter ${\bf x}_b$ of the element. Note that the above stabilization, that is quite awkward to write on paper, is instead very simple to code since it is directly based on the degree of freedom values (that is what the code operates on).

In the numerical test of Sec.~\ref{sec:num} we adopt a more advanced stabilization, 
inspired from the so called ``D-recipe'' in \cite{apollo}, to whom we refer for more details on the underlying idea.
We take 
\begin{equation}
s^E(\ub_h;(I-\Pi)\ub_h,(I-\Pi)\ub_h) = \sum_{i=1}^{\# dofs} \alpha_i(\ub_h) \:
\Big({\rm dof}_i (\ub_h-\Pi\ub_h) \Big) \Big({\rm dof}_i (\vb_h-\Pi\vb_h) \Big)\,,
\label{eqn:stabD}
\end{equation}
with the positive scalars 
$$
\alpha_i(\ub_h) = \max{ \{\alpha(\ub_h), M_{ii} \} }
$$ 
where $\alpha(\ub_h)$, defined in \eqref{eq:alfa}, is computed at the previous converged incremental step, and $M_{ii}$ denotes the $i^{th}$ diagonal entry of the consistent tangent matrix (which is the consistent tangent matrix associated to the first term in \eqref{eq:form}) computed in the centroid of the element at the previous converged incremental step.
Finally, we note that other options for $s^E$ could also be adopted and that the method is quite robust in this respect.

Let $V_h^0$ represent the subspace of $V_h$ including the Dirichlet-type boundary conditions for the problem under study, that we consider homogeneous for simplicity of exposition.
Let ${\bf f}_h$ denote the piecewise polynomial function of degree $k-2$ that is given, on each element $E$, by the $L^2(E)$ projection of the volume loading ${\bf f}$ on $[\Pp_{k-2}(E)]^2$ (in the case $k=1$ we take the projection on 
$[\Pp_0(E)]^2$).
Assuming an incremental loading procedure with $M$ steps, the numerical method at each incremental step $i=1,2..,M$ is written as
\begin{equation}\label{eq:main}
\left\{
\begin{aligned}
& \textrm{Find } \ub^i \in V_h^0 \textrm{ such that } \\
& \sum_{E \in \Omega_h} a_h^E(\ub_h^i,\vb_h)  = \frac{i}{M} \int_E {\bf f}_h \cdot \vb_h
\qquad \forall \vb_h \in V_h^0 
\end{aligned}
\right.
\end{equation}
and the final solution to the problem taken as $\ub_h = \ub_h^M$. At each incremental step, a standard Newton-Raphson procedure is applied to solve the nonlinear problem. 

Also the right hand side above is computable for $k \ge 2$ using the volume (internal to elements) degrees of freedom, while for $k=1$ the above integral can be approximated by a vertex-based quadrature rule \cite{volley,autostoppisti,BLM,PartI} (see also the notes below). 

An analogous construction is applied in the presence on non-homogeneous boundary conditions (such as boundary tractions or enforced displacements), that are also treated incrementally. A series of remarks are in order. 

\smallskip\noindent
\emph{Alternative calculation of $\alpha$}. For $M > 1$ the scalar $\alpha$ in \eqref{eq:stab} can also be computed at the previous converged incremental step. Indeed, since the method is not sensible to such parameter, this will anyway allow for a satisfactory update of the material scaling factor. This choice carries the advantage of leading to a simpler tangent matrix in the Newton iterations, since the derivatives of the $\alpha$ term disappear in the calculation.

\smallskip\noindent
\emph{More accurate loading}. By applying an enhanced construction for the space $V_h$ (see Remark \ref{rem:enh}) one can also compute a more accurate loading, using ${\bf f}_h$ piecewise in $[\Pp_k(E)]^2$. The modification is exactly the same as for standard straight-edged VEM and thus not detailed here.

\smallskip\noindent
\emph{Inelastic case}. The extension to the inelastic case is done exactly as in the Finite Element Method. Indeed, one simply needs to keep track of the history variables at the chosen integration points (within each element) and apply at each point the constitutive law, exactly as in FEM (see for instance \cite{BLM,PartII}).  

\smallskip\noindent
\emph{Simplicity of the curved case}. Note that (at the practical level) the above construction follows the same logic and structure as for the straight-edge case \cite{BLM,PartI,PartII}. In the code, the main difference is only the need to integrate along curved edges and on curved domains (that can be handled following the literature given above). 

\section{Numerical examples}\label{sec:num}

The present section has twofold purposes.
First we give numerical evidence for error analysis on linear elasticity benchmarks with available analytic solution. Then we present a couple of classical benchmarks from computational mechanics which entail solution of equilibrium boundary value problems on curved domains in conjunction with nonlinear inelastic material behavior. Such examples numerically show the efficiency of the proposed curved methodology in real scale structural simulations. 

In the remainder of the section, three virtual element approaches will be invoked according to the following acronyms: VEM $co$ i.e. the original curvilinear VEM of Sec.~\ref{sec:1-1}, VEM $cv$ i.e. the variant curvilinear VEM of Sec.~\ref{sec:1-2}, and VEM $s$ i.e. the straight-edge VEM presented and validated in \cite{PartI}. As already mentioned, we use the stabilization in Equation~\eqref{eqn:stabD}.

\begin{remark}\label{rem:int}{\bf Integration rules}.
In the following numerical tests, in order to compute the volume integrals on (possibly) curved elements, we adopt the same integration algorithm used in \cite{BRV_curvi} (that in turn was taken from \cite{vianello1,vianello2,vianello3}) combined with a rotation of the cartesian coordinates. 
Such algorithm allows, for any given order $n \in {\mathbb N}$ and element $E$, to generate a set of points and weights yielding an integration rule that is exact for polynomials of degree up to $n$. In the presented tests, when combined with a suitable coordinate rotation, such algorithm yielded (in all cases considered) points that are internal to the element and weights that are positive.
Nevertheless, such algorithm has the drawback of using many integration points and thus needs to be combined with a compression rule for the sake of efficiency. 
In order to keep the positivity of the weights during the compression, we follow the compression rule of \cite{Artioli-Vianello} instead of the one used in \cite{BRV_curvi}. The final number of integration points is equal to $(n+1)(n+2)/2$. An investigation on the number $n$ to be used in the VEM analysis is shown in Example 1.
{ Finally, we point out that for all edge integrals (that are standard mapped one dimensional integrals on the reference interval), one can use a classical Gau\ss-Lobatto integration rule. The minimal required number of points (in analogy with the straight edge case) for a scheme of order $k$ are $k+1$ points, including the two extrema, but a higher number of points can be used to obtain a safer rule.}
\end{remark}

\subsection{Numerical tests for elastic materials}
\label{exe:Math}

In the present section we consider 2D elasticity problems with known solution.
The first numerical test investigates the convergence of the scheme and develops some comparisons. The second test gives numerical evidence that the novel virtual element space with curved edges proposed in Sec.~\ref{sec:1-2} contains rigid body motions.

In both examples we consider as domain $\Omega$ the circle centered at the origin with unit radius, and the following meshes:
\begin{itemize}
 \item \texttt{quadN}: a mesh composed by squares and a polygon at the center, see Fig.~\ref{fig:mesh} (a);
 \item \texttt{rhexN}: a mesh composed by regular hexagons inside the domain and arbitrary polygons close to the boundary, 
 see Fig.~\ref{fig:mesh} (b);
 \item \texttt{voroN}: a mesh composed by Voronoi cells, see Fig.~\ref{fig:mesh} (c).
\end{itemize}
All of the above polygons abutting the circle boundary have curved edges. Here \texttt{N} refers to the number of polygons in the mesh.
The last two meshes are generated with \texttt{Polymesher}~\cite{PolyMesher}.


\begin{figure}[!htb]
\begin{center}
\begin{tabular}{ccc}
\includegraphics[width=0.31\textwidth]{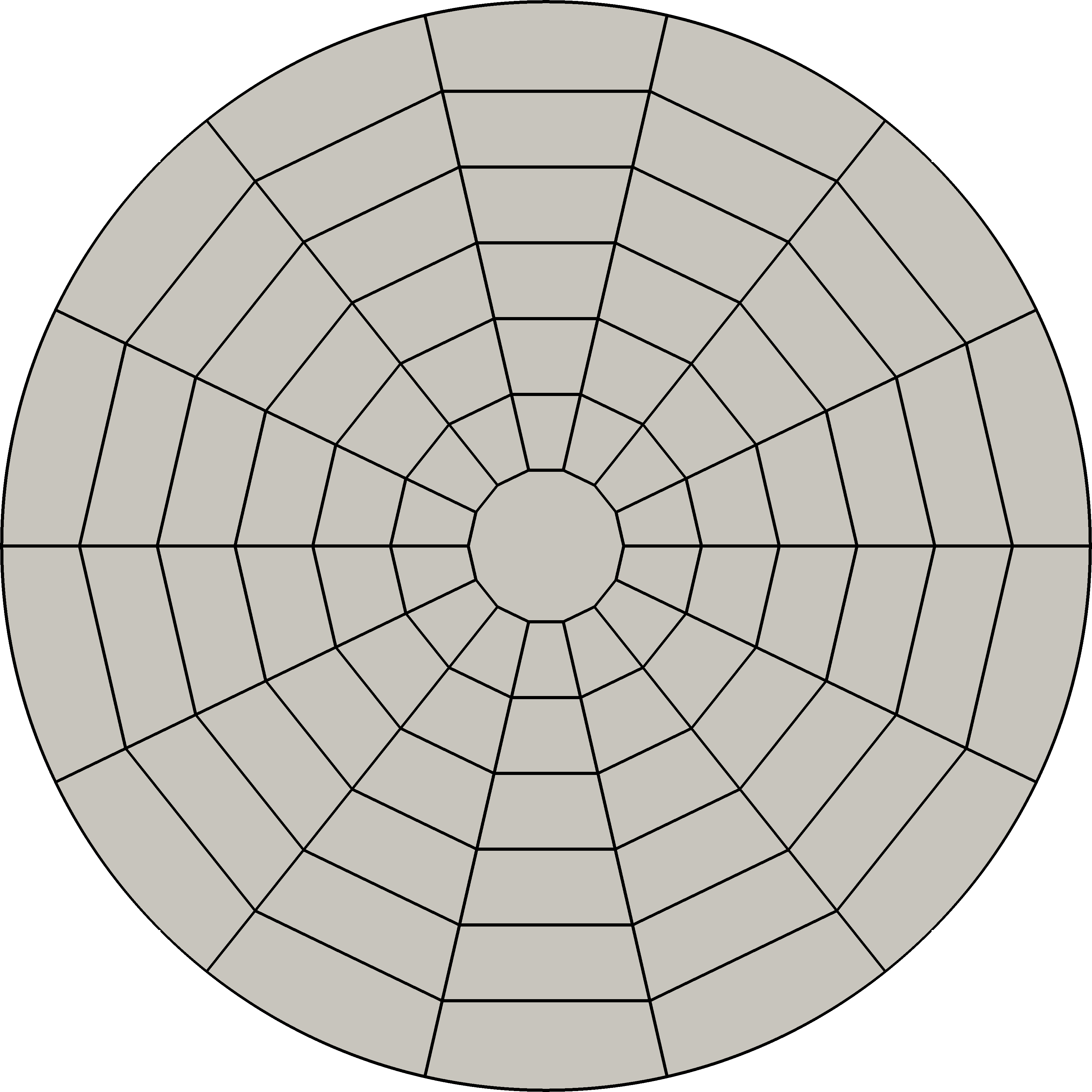} &
\includegraphics[width=0.31\textwidth]{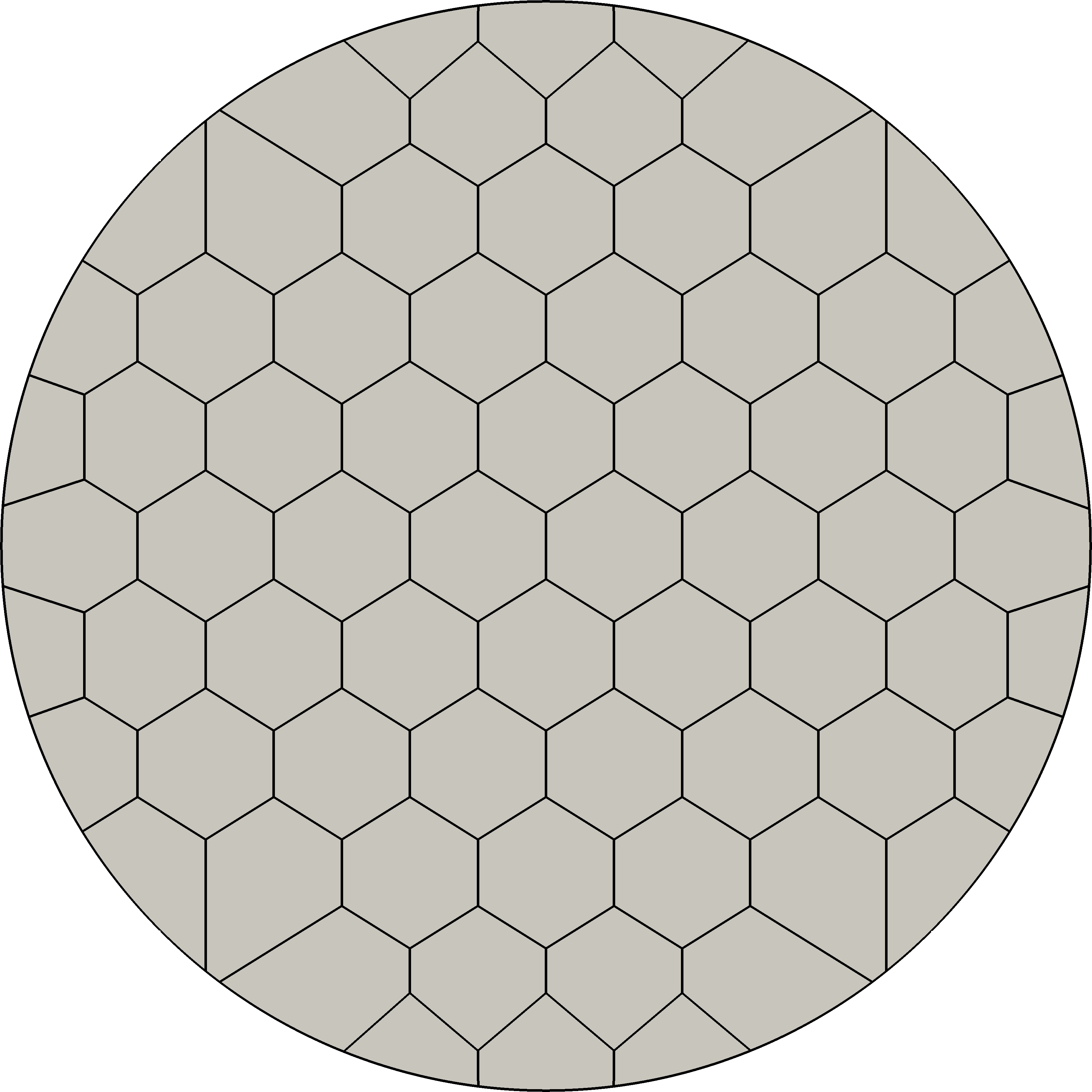} &
\includegraphics[width=0.31\textwidth]{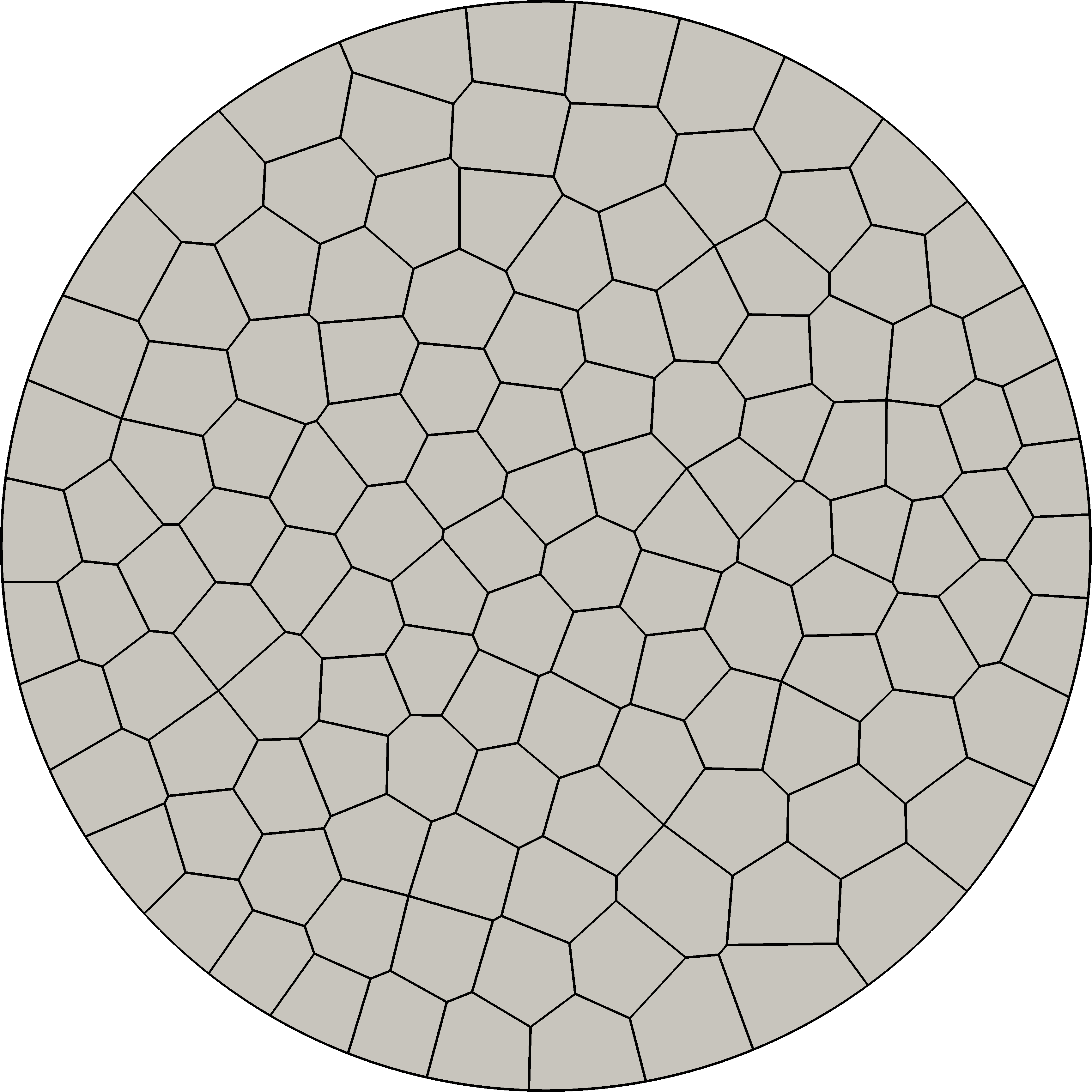} \\
(a) &(b) &(c)
\end{tabular}
\end{center}
\caption{Meshes used in Examples 1 and 2: (a) \texttt{quad80}, (b) \texttt{rhex70} and (c) \texttt{voro128}.}
\label{fig:mesh}
\end{figure}

In order to develop a convergence analysis, in the numerical tests we re-define the mesh-size $h$ as the mean value of element diameters, i.e.
$$
h = \frac{1}{N}\sum_{E\in \Omega_h} h_E\,,
$$
where $h_E$ is the diameter of element $E\in\Omega_h$.
For each mesh type, we consider a set of uniformly refined meshes and adopt the following error measures on displacement and strain field, respectively:
\begin{itemize}
 \item \textbf{$l^\infty$-error on displacement field:} 
 $$
 e^{\ub} := \frac{\|\ub-\ub_h\|_{l^\infty(\mathcal{N}_h)}}{\|\ub\|_{l^\infty(\mathcal{N}_h)}}\,,
 $$
 with the norm 
 $$
 \|\ub-\ub_h\|_{l^\infty(\mathcal{N}_h)}:=\max_{\xb \in \mathcal{N}_h} \|\ub(\xb)-\ub_h(\xb)\|\,
 $$
defined through the values of the solution on the mesh skeleton, i.e. at the vertex/edge nodes $\xb \in \mathcal{N}_h$, and where $\|\cdot\|$ is the Euclidean norm. Although we do not have a rigorous theoretical analysis for this norm, in analogy with Finite Elements the expected trend of such error is $O(h^{k+1})$ for VEM elements of order $k$ and sufficiently regular solutions.
\item \textbf{$L^2$-error on strain field:} in this case we exploit the projection operator $\Pi^{\epsb}$ and we compute
 $$
 e^{\epsb} := \frac{\|\epsb(\ub)-\Pi^{\epsb}\ub_h\|_{L^2(\Omega_h)}}{\|\epsb(\ub)\|_{L^2(\Omega_h)}}\,.
 $$
On the basis of the interpolation results of Proposition \ref{prop:interp} the expected trend of such error is $O(h^k)$ for VEM elements of order $k$ and sufficiently regular solutions.
\end{itemize}

\paragraph{Example 1. Convergence analysis.}

In this example we develop a convergence analysis for the proposed VEM curved elements, also comparing integration rules of different order.
Moreover, we compare the results obtained via the proposed approach 
(which can deal with curved elements and thus reproduce exactly the geometry of interest) and the standard VEM that uses straight polygonal meshes (and thus the geometry boundary is approximated by a piecewise linear curve). In the present tests we do not distinguish between the scheme VEM $co$ of Sec.~\ref{sec:1-1} and the scheme VEM $cv$ of Sec.~\ref{sec:1-2} since they yielded almost identical results.

We consider a nonlinear elastic material, characterized by the 
Hencky-von Mises constitutive law~\cite{lawPrimo}
$$
\sigmab(x, \nabla \ub) := 
\tilde{\lambda}(\text{dev}(\epsb(\ub)))\,\text{tr}(\epsb(\ub))\,\textbf{I} + 
\tilde{\mu}(\text{dev}(\epsb(\ub)))\,\epsb(\ub)
$$
where
$$
\tilde{\mu}(\rho) :=  \frac{3}{4}(1+(1+\rho^2)^{-1/2})\cdot10^4 \text{MPa}\,\qquad\text{and}\qquad
\tilde{\lambda}(\rho) := \frac{3}{4}(1-2\tilde{\mu}(\rho))\cdot10^4 \text{MPa}\qquad\forall\rho\in\mathbb{R}^+
$$
are the non-linear Lam\'e functions, $
\epsb(\ub) := \frac{1}{2}(\nabla \ub + (\nabla \ub)^t)$
is the small deformation strain tensor, $\text{tr}(\taub)$ is the trace operator and 
$$
\text{dev}(\taub) := \left\|\left(\taub - \frac{1}{2}\text{tr}(\taub)\textbf{I}\right)\right\|_{\mathcal{F}}\,,
$$
is the Frobenius norm of the deviatoric part of the tensor $\taub$.

The volume force density field and the homogeneous Dirichlet boundary conditions are chosen in such a way that the solution of Problem~\ref{eq:con-pbl}
$$
\ub(x,\,y) = \left( \begin{array}{r} 
                     \sin(\pi(x^2+y^2))\\
                    2\cos(\pi(x^2+y^2))
                    \end{array}\right)\,.
$$

The problem is solved with two VEM approaches: 
the VEM co for curved geometry described in Sec.~\ref{sec:1-1} and the standard VEM s displacement-based VEM method~\cite{volley,BLM,PartI,PartII} with straight edges (that approximates the domain of interest by a straight edge polygonal). We report numerical evidence for method degrees $k=1,2,3$.

In order to investigate the effect of the adopted integration rule, we first compare two different options for quadrature. For the \emph{minimal} quadrature rule we use an integration rule of order $n=2k-2$ for volume integrals (see Remark \ref{rem:int}), that is the minimal one that would guarantee exact integration for a linear elastic problem with constant coefficients. For the \emph{higher-order} quadrature rule we instead use an integration rule of order $n=2k$, that is more accurate but more expensive (expecially in view of using the scheme for inelastic problems).
{ Analogously, for edge integrals we use a Gau\ss-Lobatto rule with $k+1$ points in the \emph{minimal} quadrature case, and $k+2$ points in the \emph{higher-order} case.}

In Figs.~\ref{fig:exe1U} and~\ref{fig:exe1Eps} we show the convergence plots in terms of the above introduced error norms for the curved VEM co elements, comparing the minimal and the higher-order rules. We can see that in all cases the convergence slopes are the expected optimal ones ($O(k^{k+1})$ for nodal displacement error and $O(h^k)$ for the strain $L^2$ error). When considering the displacement error, the higher-order rule yields slightly more accurate (although of the same convergence rate) results. Our conclusion is that, in general, the minimal rule seems the better choice since it provides similar results with fewer quadrature points.

In Fig.~\ref{fig:exe1nogeo} we depict the error curves for the straight VEM s elements, that is without an exact geometry approximation. We use the high-order quadrature rule. Due to the error introduced by the geometry approximation, for $k \ge 2$ one obtains a sub-optimal convergence rate. When comparing the results of Fig.~\ref{fig:exe1nogeo} with those for curved elements in Figs.~\ref{fig:exe1U} and \ref{fig:exe1Eps} the advantage of having an exact geometry approximation appears clearly, expecially for higher values of $k$ and finer meshes.

\begin{figure}[!htb]
\begin{center}
\begin{tabular}{cc}
\includegraphics[height=0.4\textwidth]{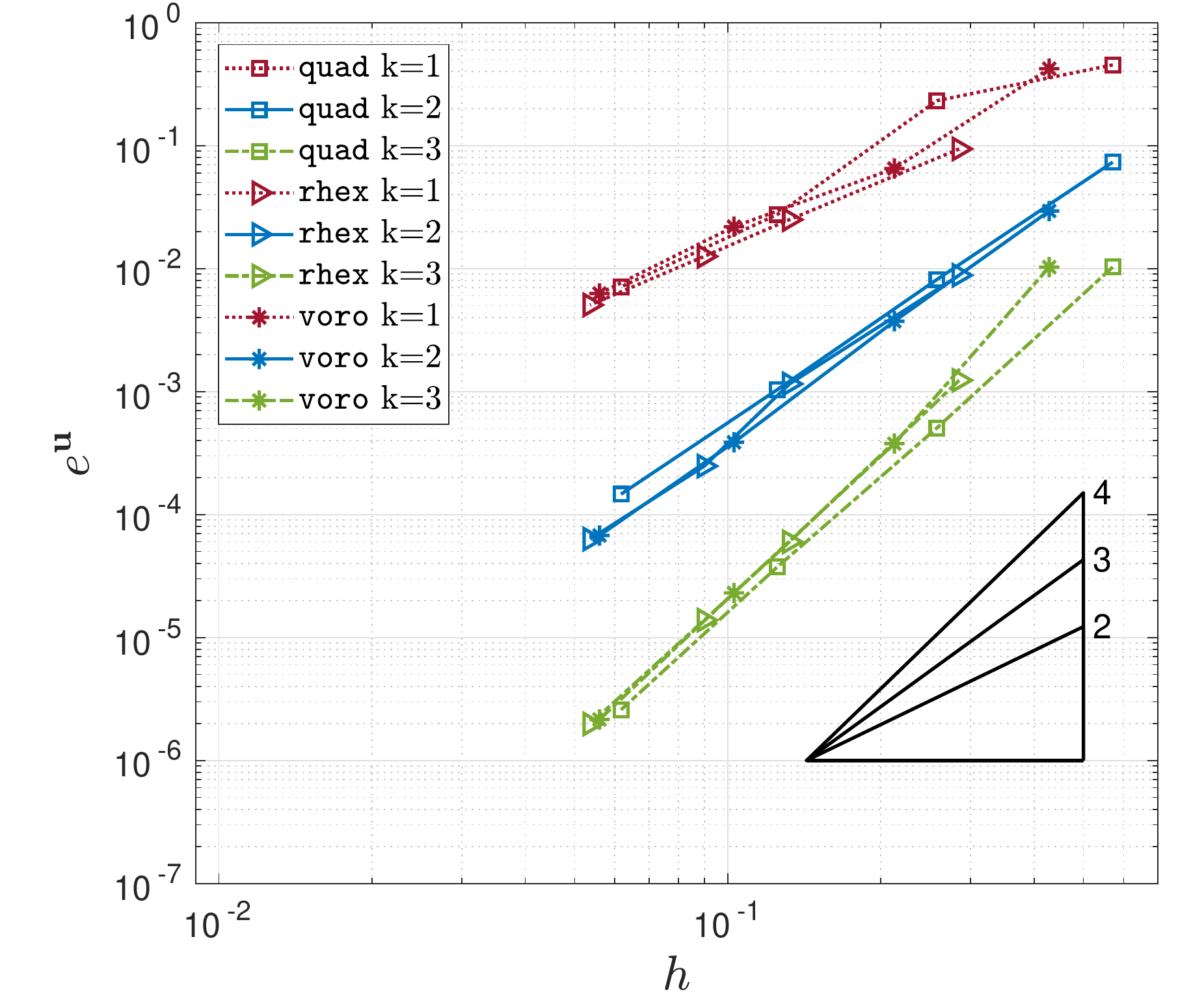} &
\includegraphics[height=0.4\textwidth]{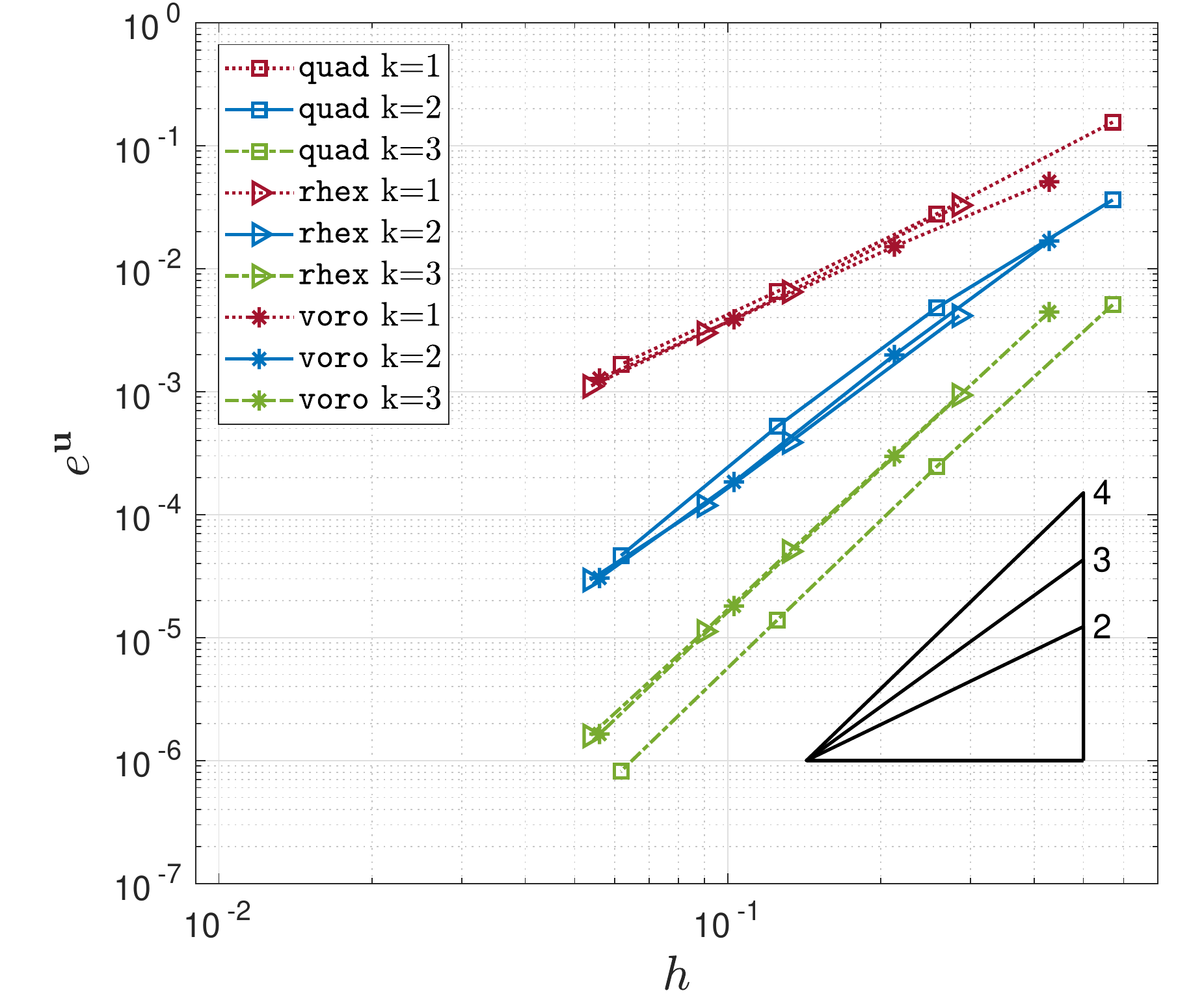}\\
(a) & (b)
\end{tabular}
\end{center}
\caption{{Example 1 (curved VEM co elements): convergence plot for $e^\ub$
with minimal quadrature (a), and higher order quadrature, (b).}}
\label{fig:exe1U}
\end{figure}

\begin{figure}[!htb]
\begin{center}
\begin{tabular}{cc}
\includegraphics[height=0.4\textwidth]{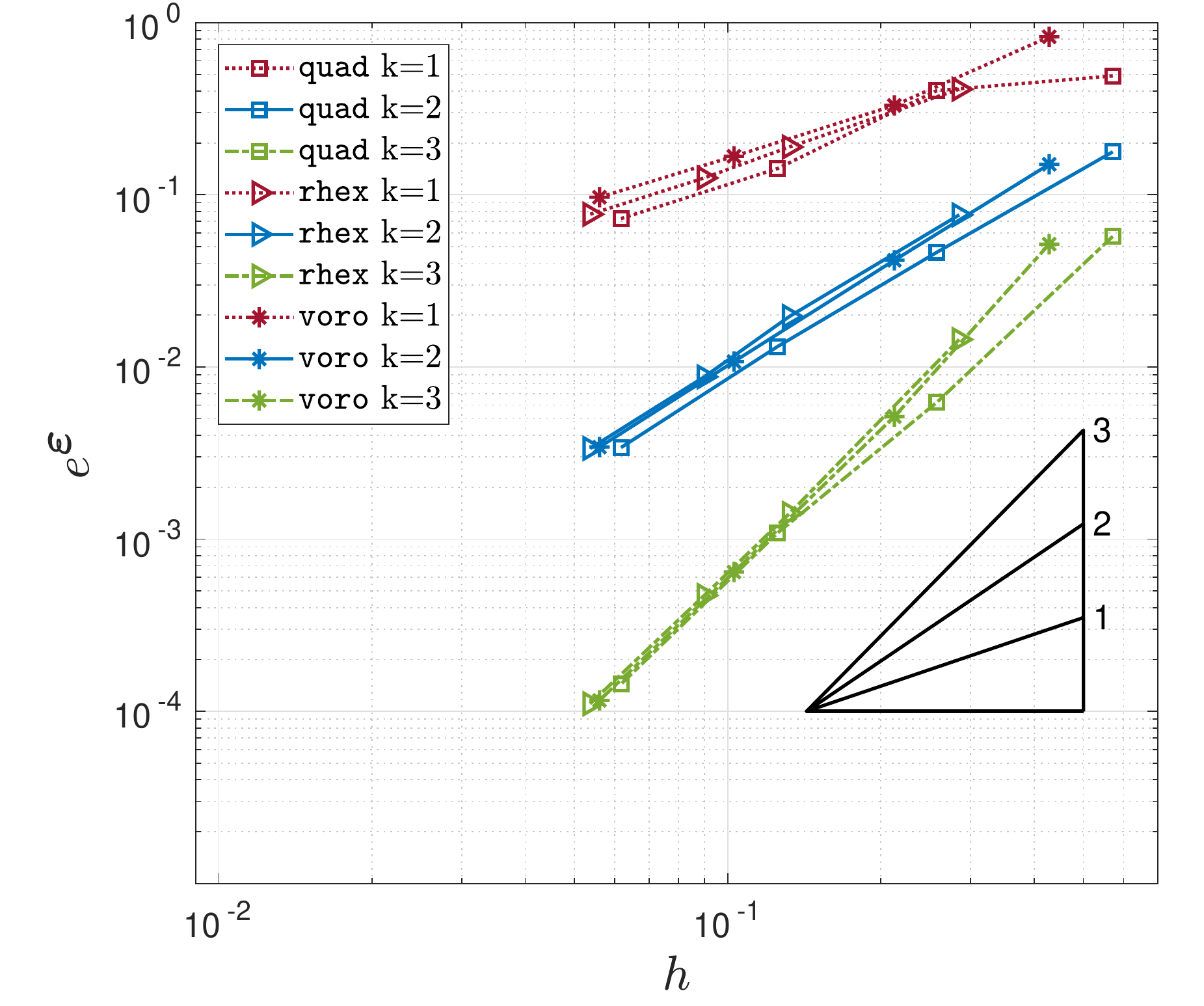} &
\includegraphics[height=0.4\textwidth]{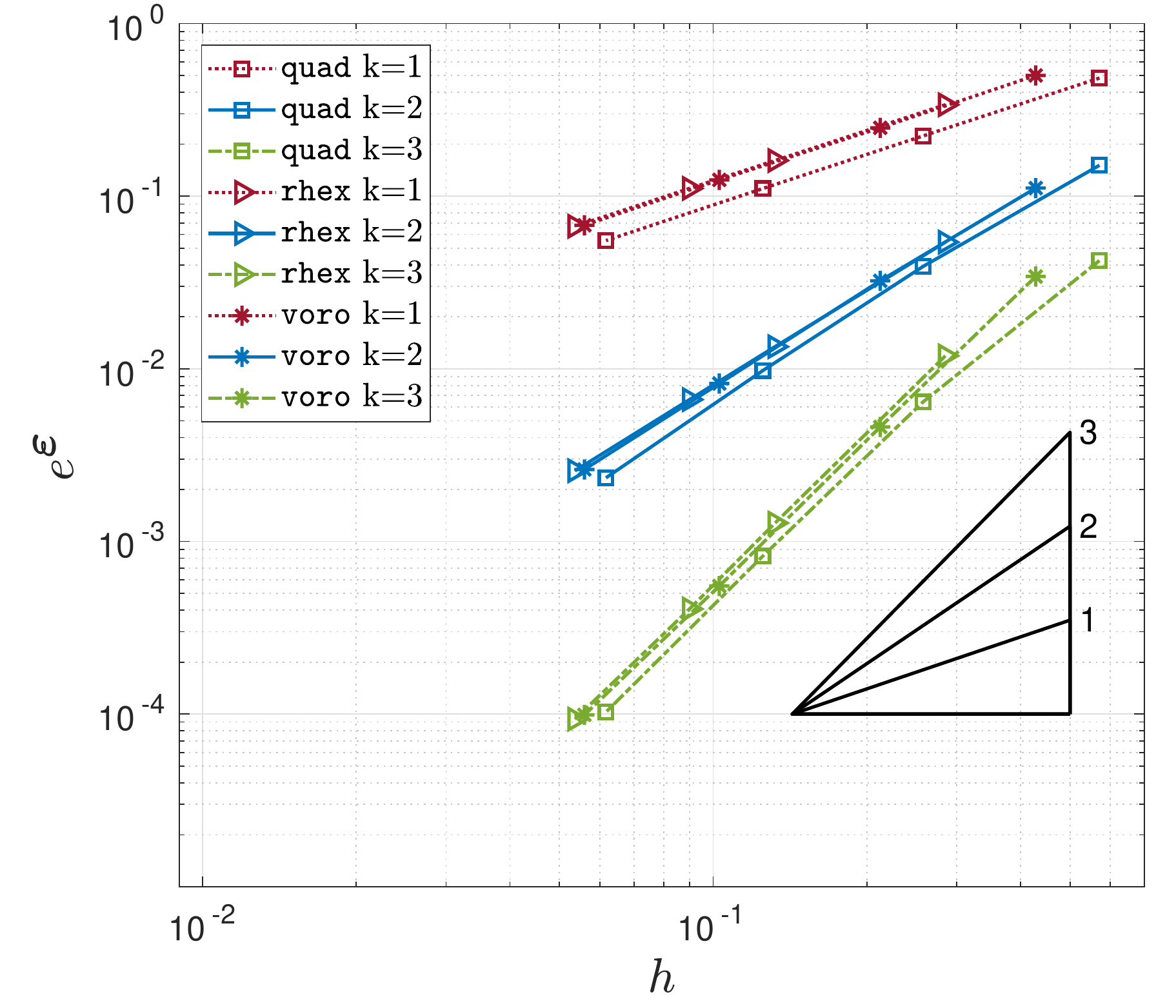}\\
(a) & (b)
\end{tabular}
\end{center}
\caption{{Example 1 (curved VEM co elements): convergence plot for $e^{\epsb}$ with minimal quadrature (a), and higher order quadrature, (b).}}
\label{fig:exe1Eps}
\end{figure}

\begin{figure}[!htb]
\begin{center}
\begin{tabular}{cc}
\includegraphics[height=0.4\textwidth]{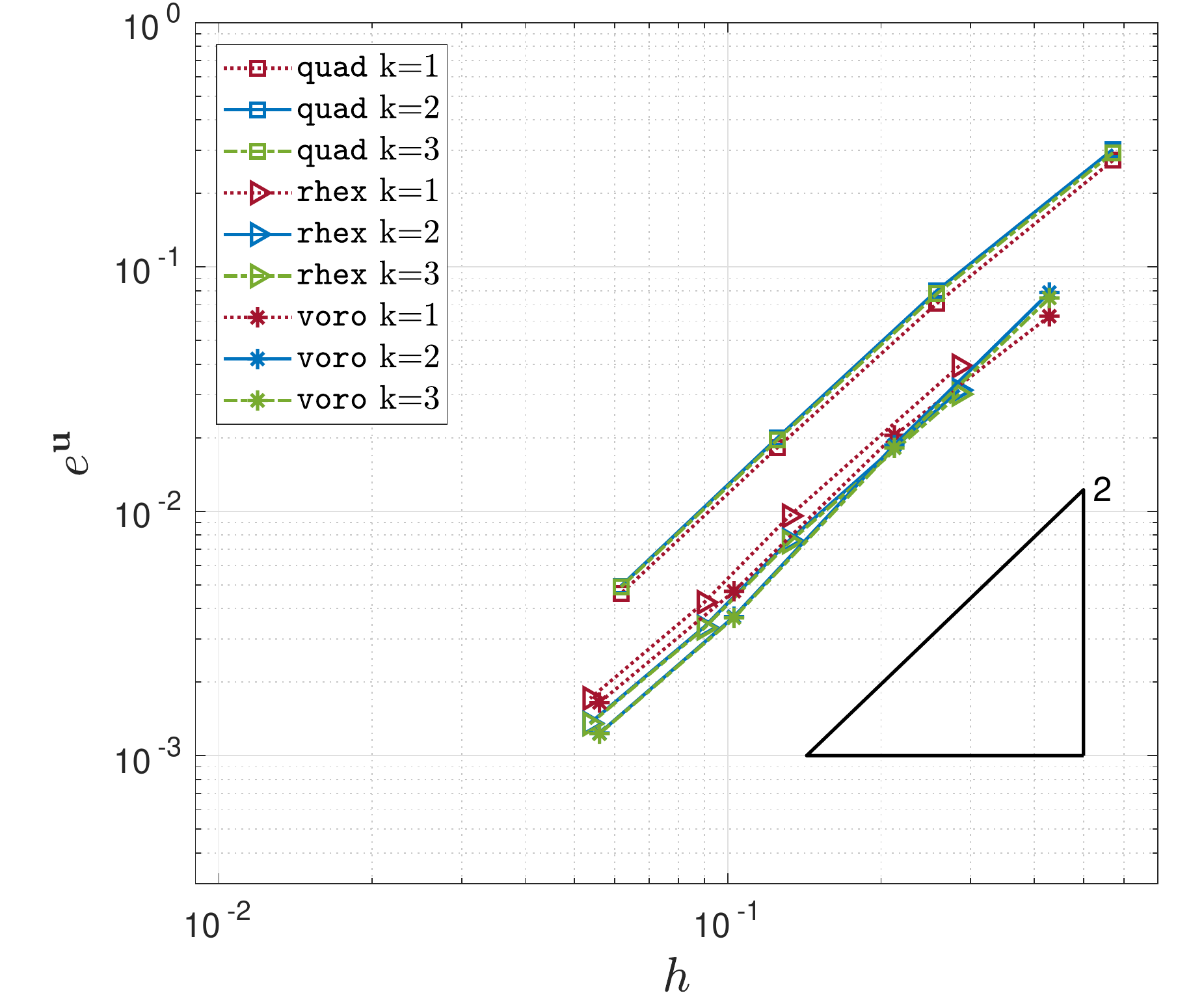} &
\includegraphics[height=0.4\textwidth]{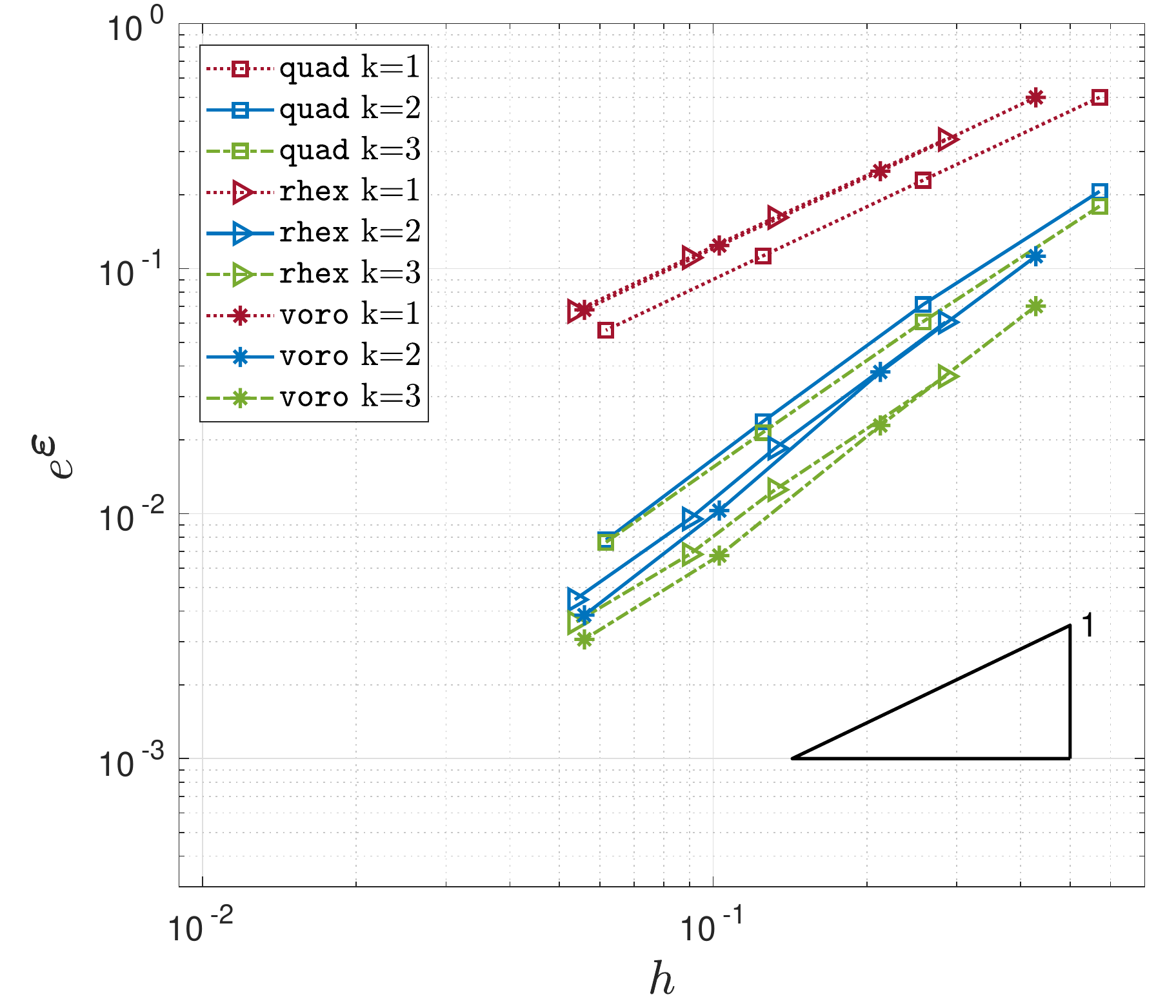}\\
(a) & (b)
\end{tabular}
\end{center}
\caption{{Example 1 (straight VEM s elements): convergence with high order quadrature, error $e^\ub$ (a) and error plot $e^{\epsb}$ (b).}}
\label{fig:exe1nogeo}
\end{figure}

\paragraph{Example 2. Rigid body motions.}

In this paragraph we compare the virtual element approach VEM $co$ of Sec.~\ref{sec:1-1}, and its variant VEM $cv$ of Sec.~\ref{sec:1-2}.
The main difference between these two methods relies on the inclusion of the rigid body motion in the virtual displacement space.
Indeed, both of them are defined for curved elements and they will present an optimal convergence rate, but \emph{only} the latter method VEM $cv$ will guarantee exact rigid body motion representation in the solution.

To give a numerical evidence of this feature, we consider a linear elastic problem, where the volume force density is zero.
As in the previous test, we consider a unit radius circle centered at the origin, with Dirichlet boundary conditions set in such a way that the exact solution is the vector field (representing a rigid body motion)
\begin{equation}
\ub(x,\,y) = \left( \begin{array}{r} 
                     -y\\
                      x
                    \end{array}\right)\,.
\label{eqn:solExe2}
\end{equation}
All volume and edge integrals are computed with a very high quadrature rule in this test, in order to eliminate the effects of numerical quadrature in the comparison.

In Table~\ref{tab:exe2PatchU}  we report the $e^\ub$ error, $k=1,2,3$, for the two curvilinear VEM variants and some sample meshes. While it is evident that the VEM $cv$ is exact up to machine precision, the error levels associated with the VEM $co$ are much higher. This shows that the virtual element space in the new variant (VEM $cv$) contains the solution of the problem at hand, 
see Equation~\eqref{eqn:solExe2},
while the standard space (VEM $co$) does not contains such solution.


\begin{table}[!htb]
\begin{center}
\begin{tabular}{|c|cc|cc|cc|}
\cline{2-7}
\multicolumn{1}{c}{}&
\multicolumn{2}{|c|}{\texttt{quad481}} &
\multicolumn{2}{c|}{\texttt{rhex626}} &
\multicolumn{2}{c|}{\texttt{voro1800}}\\
\cline{2-7}
\multicolumn{1}{c|}{}&VEM $co$ &VEM $cv$ &VEM $co$ &VEM $cv$ &VEM $co$ &VEM $cv$\\
\hline
$k=1$ &2.2018e-03 &1.0847e-15 &4.0851e-04 &3.2530e-13 &4.4464e-04 &1.0139e-15 \\
$k=2$ &1.0379e-06 &5.4964e-15 &4.1290e-08 &1.6785e-13 &4.3888e-08 &6.9502e-15 \\
$k=3$ &5.6965e-08 &2.4372e-14 &1.9404e-09 &9.8653e-14 &1.6537e-09 &1.0776e-14 \\
\hline
\end{tabular} 
\end{center}
\caption{Example 2 (rigid body motions): error $e^{\ub}$ for different mesh types and VEM approximation degree.}
\label{tab:exe2PatchU}
\end{table}

\subsection{Numerical tests for inelastic materials}\label{exe:Eng}

In this section we consider two classical benchmark examples involving inelastic materials.
Since we do not have a reference solution for such problems, we use the software FEAP on a very fine mesh to validate the results.

\paragraph{Example 3. Thick-walled viscoelastic cylinder subjected to internal pressure.}
\label{ss:viscocyl}
We consider a classical numerical test regarding a thick-walled cylinder characterized by a viscoelastic constitutive response~\cite{Zienkiewicz_Taylor_Fox}.  The cylinder has inner [resp. outer] radius $R_i = 2$ [$R_o=4$] and is subjected to uniform pressure $p$ on the inner surface, see Fig. \ref{fig:pressure_cylinder_geom}. The material is isotropic and obeys a viscoelastic constitutive law as outlined Sec. 3.1 of reference \cite{PartII}. The material properties are set assuming $M=1$ and $\lambda_1 \equiv  \lambda = 1$, i.e. a {\it standard linear solid} is considered \cite{Zienkiewicz_Taylor_Zhu13}. Young's modulus and Poisson's ratio are $E = 1000$, $\nu = 0.3$, respectively. Two sets of viscoelastic parameters are adopted for the present analysis, i.e. $\left( \mu_0 , \mu_1 \right)_{\veOne} = \left( 0.01, 0.99 \right )$ and, $\left (\mu_0 , \mu_1 \right)_{\veTwo} = \left( 0.3, 0.7 \right )$, respectively. 
The former case is devised in a way that the bulk-to-shear moduli ratio for instantaneous loading is given by $K/G(0) = 2.167$ and for long time loading, say at $t=8$, by $K/G(8) = 216.7$, which indicates a nearly incompressible behavior for sustained loading (for instance, at $t = \infty$ the Poisson ratio results $0.498$). The second material set indicates an intermediate response at $t = \infty$ after loading is  applied. Plane strain assumption is applied in this numerical simulation.

For symmetry, only a quarter of the cylinder cross section is meshed, as reported in Fig.~\ref{fig:meshesExe3}, where structured convex curvilinear quadrilaterals (a), and unstructured convex quadrilaterals (b) meshes are portrayed, respectively. Zero normal displacement is enforced along straight edges. The structural response for an internal pressure $p = 10$ applied at $t=0$ and kept constant until $t=20$ is computed through $20$ equally spaced time instants by using the generalized Maxwell model in Prony series form, which, albeit intrinsically three-dimensional, translates naturally in the present two-dimentional context \cite{Zienkiewicz_Taylor_Fox}. 

A displacement accuracy analysis is drawn between the proposed curved VEM $co$/$cv$ formulations, for $k=1,2,3$, with a reference solution reported for comparison. Such reference solution is obtained with quadratic quadrilateral displacement based finite elements with nine nodes $Q9$ running in the FEAP platform \cite{Zienkiewicz_Taylor_Fox} using an extremely fine mesh.
Since the minimal quadrature rule was shown in the previous section to be able to grant the correct convergence rates and is cheaper in terms of computational costs, we use such rule in the present tests.

The integration-step versus displacement curves for control points A and B (see Fig. \ref{fig:pressure_cylinder_geom}) is shown in Fig. \ref{fig:visco_disp} (a)-(b) for the compared solutions together with the reference one, for the two material parameter sets ${\veOne}$ and ${\veTwo}$ introduced above. We here show only the results for the (more demanding) distorted meshes (b), the ones for other case being analogous. It is observed that compared VEM formulations present a response in excellent agreement with the reference solution also for non structured and distorted discretizations of the domain, thus proving the efficiency of the proposed formulation.

{
We consider a quadrilateral mesh and we take a subset of quadrature points deployed along a radial direction, see the quadrature points high-lighted with crosses in Fig.~\ref{fig:sigma_rho}~(a).
According to the theory proposed in~\cite{Zienkiewicz_Taylor_Fox}, 
the radial stress component, $\sigma_{\rho}$, has a regular trend from -10 to 0, 
see Fig.~4.5~(b) of~\cite{Zienkiewicz_Taylor_Fox}.
In such example, standard displacement based finite elements often display 
an irregular trend characterized by spurious oscillations \cite{Zienkiewicz_Taylor_Fox}.
On the contrary, the proposed virtual element method results free from spurious oscillations and it is better aligned with the reference trend,   
compare Fig.~\ref{fig:sigma_rho}~(b) of the present paper and Fig.~4.5~(b) of~\cite{Zienkiewicz_Taylor_Fox}.}

\begin{figure}
\begin{center}
\includegraphics[bb=0 0 400 800, angle=0, scale=0.30]{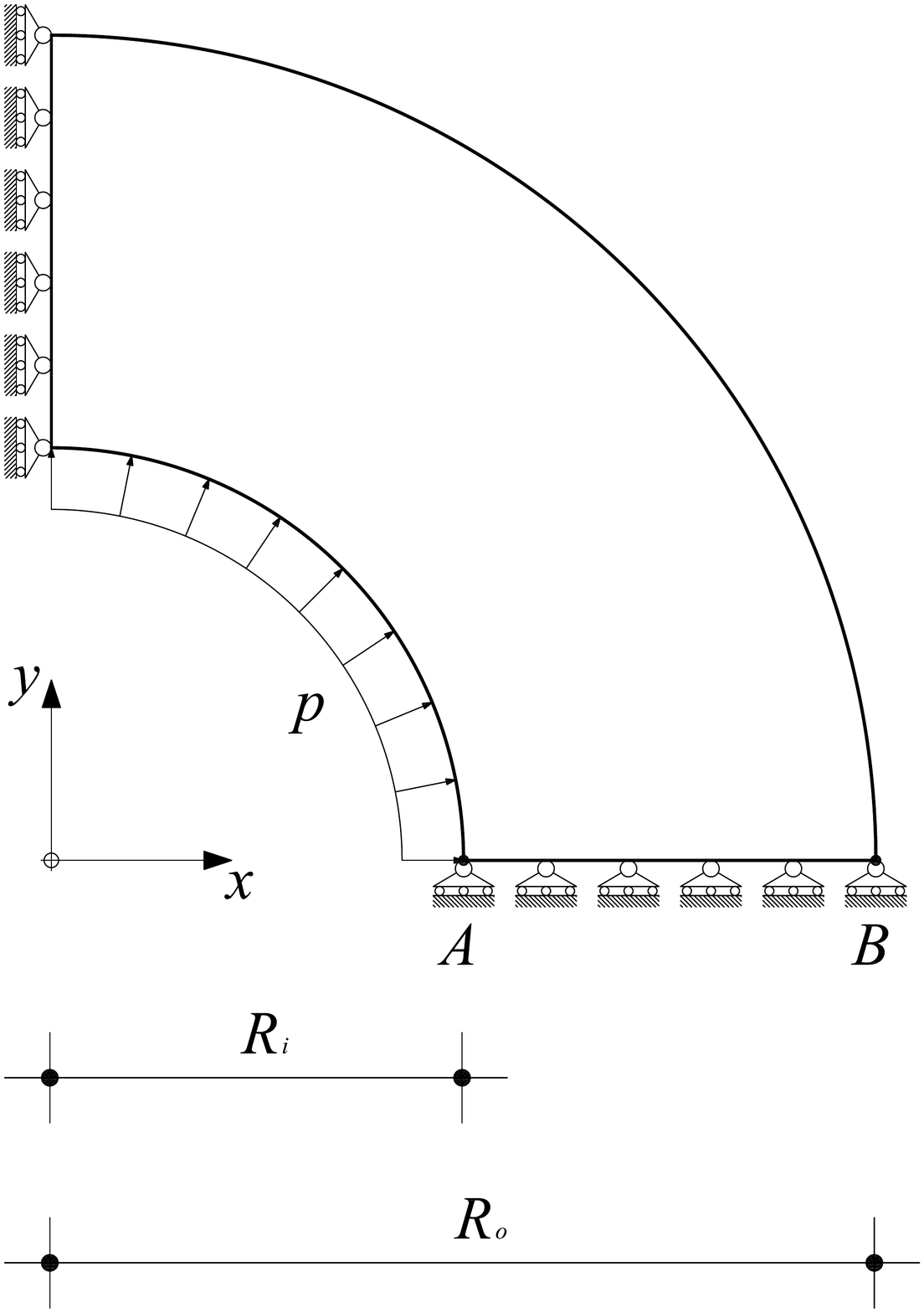}
\end{center}
\caption{Example 3. Thick-walled viscoelastic cylinder subjected to internal pressure. Geometry, boundary conditions, applied load.}
\label{fig:pressure_cylinder_geom}
\end{figure}

\begin{figure}[!htb]
\begin{center}
\begin{tabular}{cc}
\includegraphics[width=0.48\textwidth]{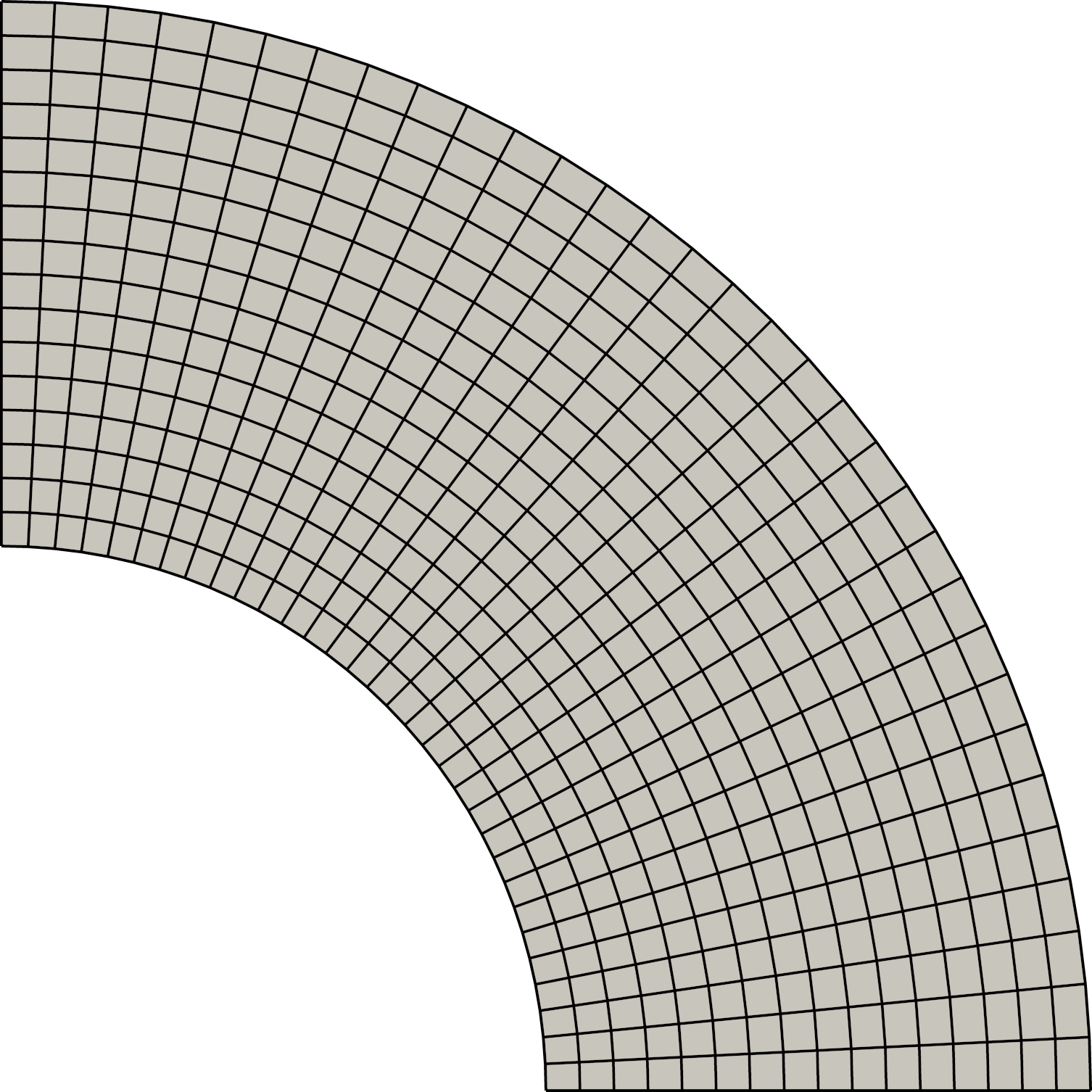} &
\includegraphics[width=0.48\textwidth]{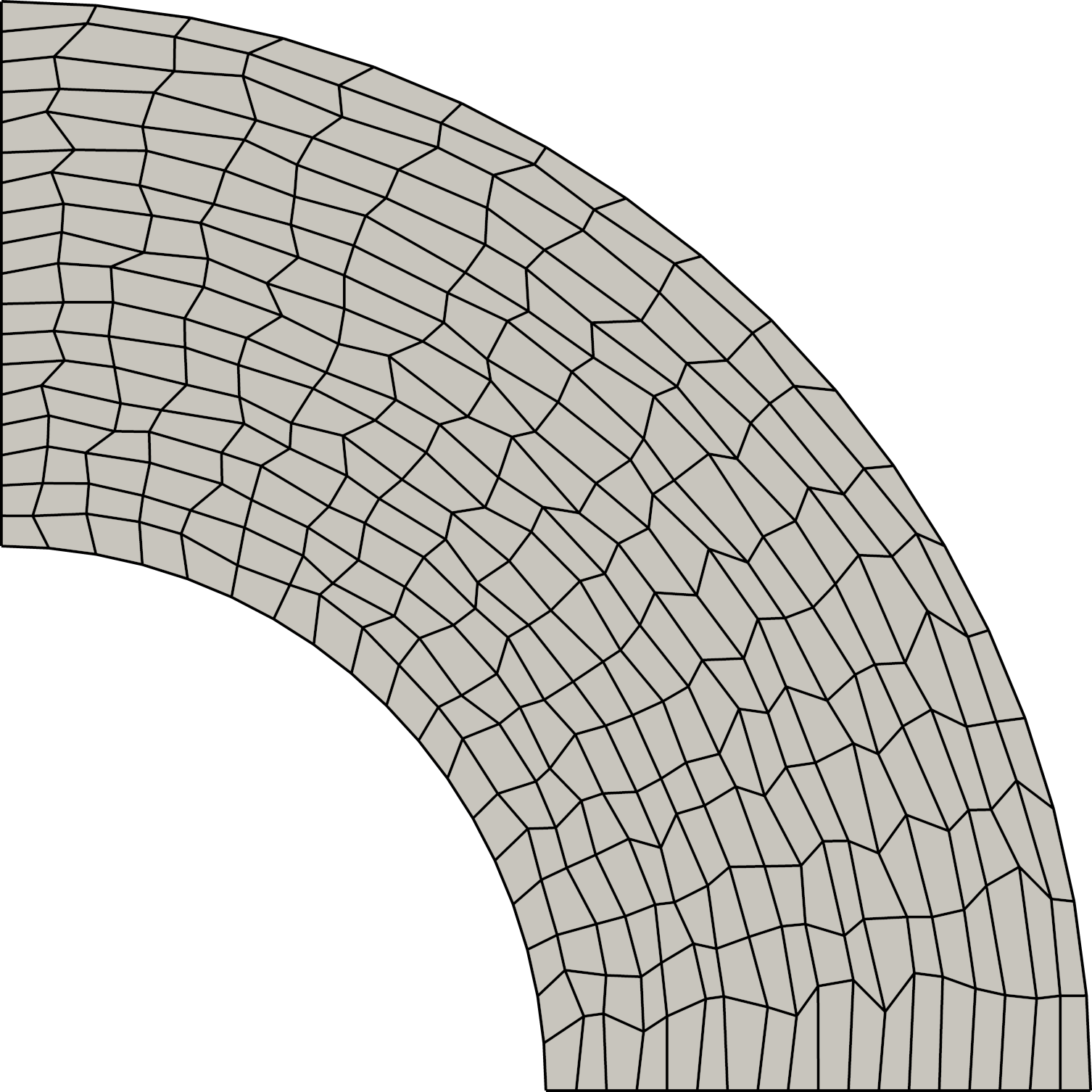} \\
& \\
(a) & (b)
\end{tabular}
\end{center}
\caption{Example 3. Thick-walled viscoelastic cylinder subjected to internal pressure. (a) Structured quadrilateral mesh composed by 272 elements. (b) Unstructured quadrilateral mesh composed by 324 elements.}
\label{fig:meshesExe3}
\end{figure}

\begin{figure}
\begin{center}
\begin{tabular}{cc}
\includegraphics[width=0.48\textwidth]{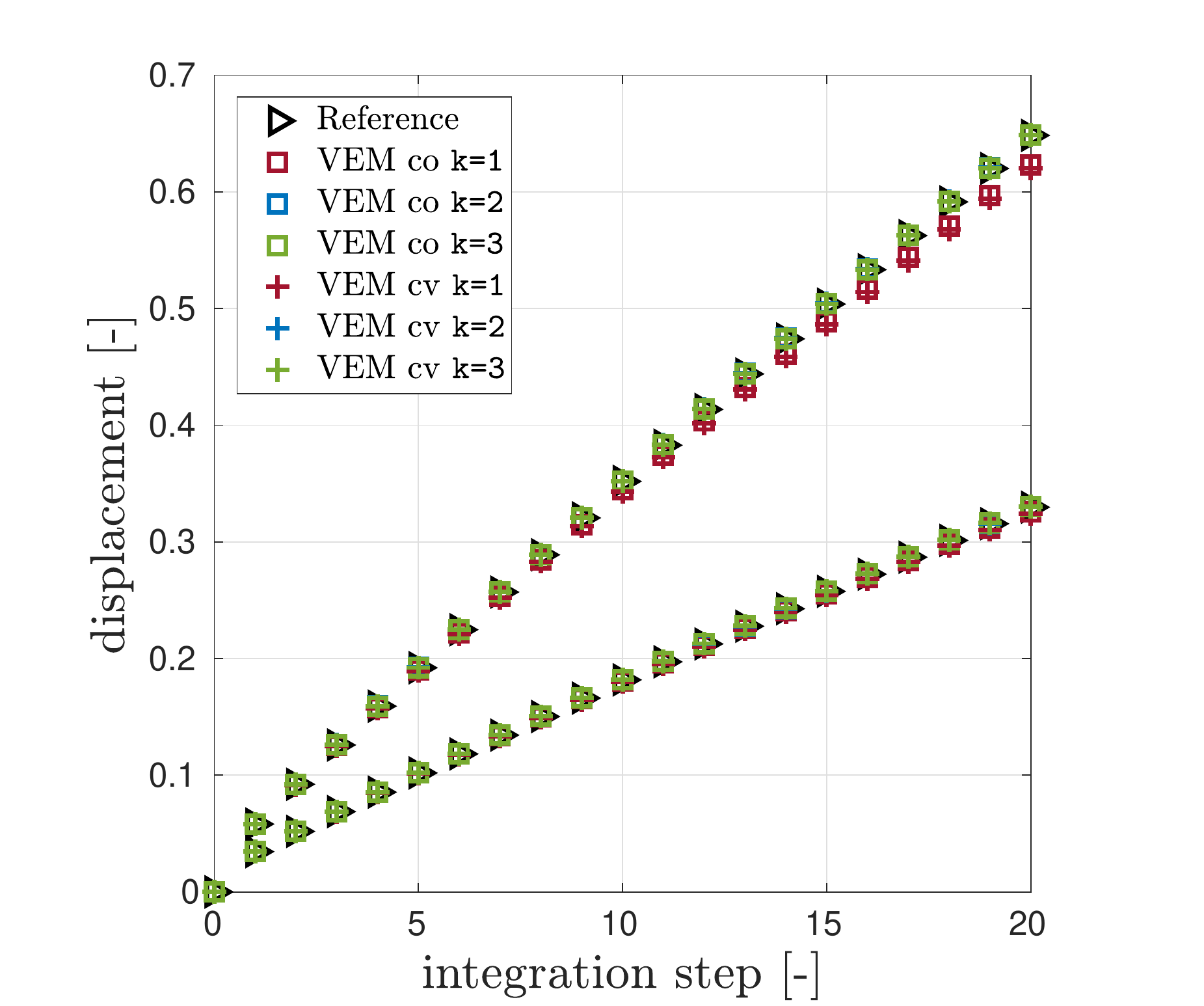} &
\includegraphics[width=0.48\textwidth]{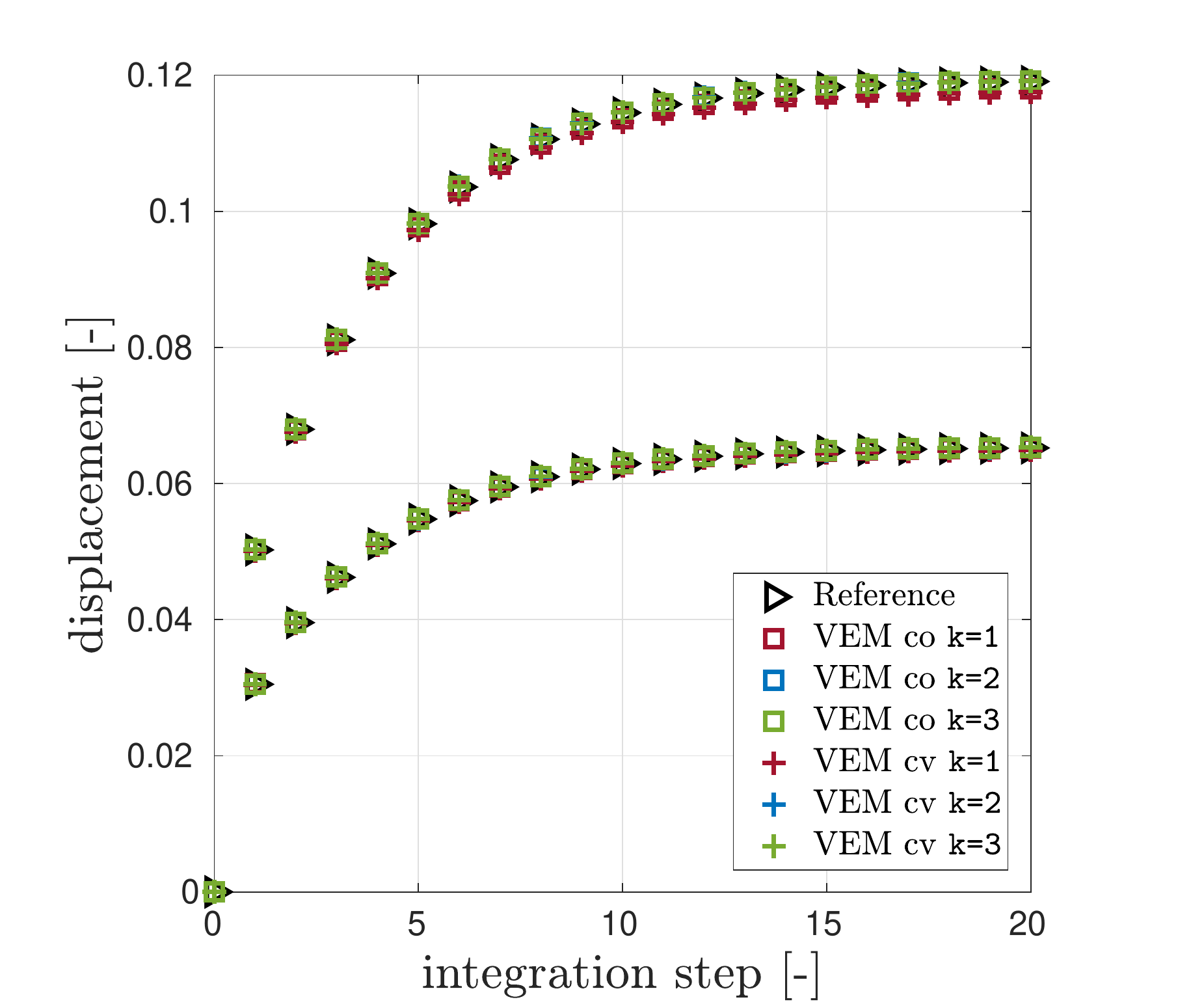}\\
(a) & (b)
\end{tabular}
\end{center}
\caption{Example 3. Thick-walled viscoelastic cylinder with internal pressure. Integration step {\it vs.} radial displacement curves for control points $A$ (higher curve), and $B$ (lower curve). (a) case $\left( \mu_0 , \mu_1 \right)_{\veOne} = \left( 0.01, 0.99 \right )$; (b) case $\left (\mu_0 , \mu_1 \right)_{\veTwo} = \left( 0.3, 0.7 \right )$.}
\label{fig:visco_disp}
\end{figure}

\begin{figure}
\begin{center}
\begin{tabular}{cc}
\includegraphics[width=0.48\textwidth]{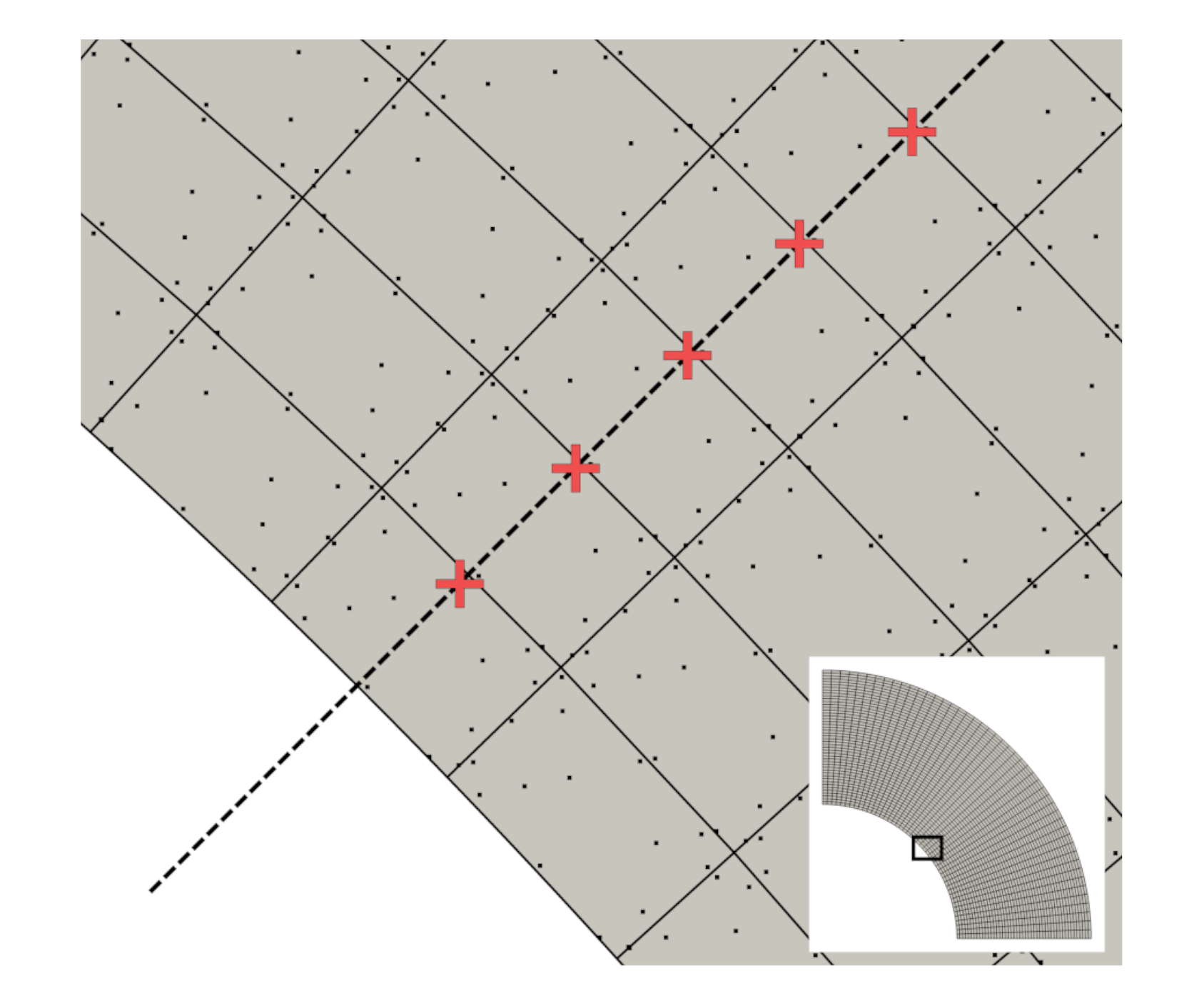} &
\includegraphics[width=0.48\textwidth]{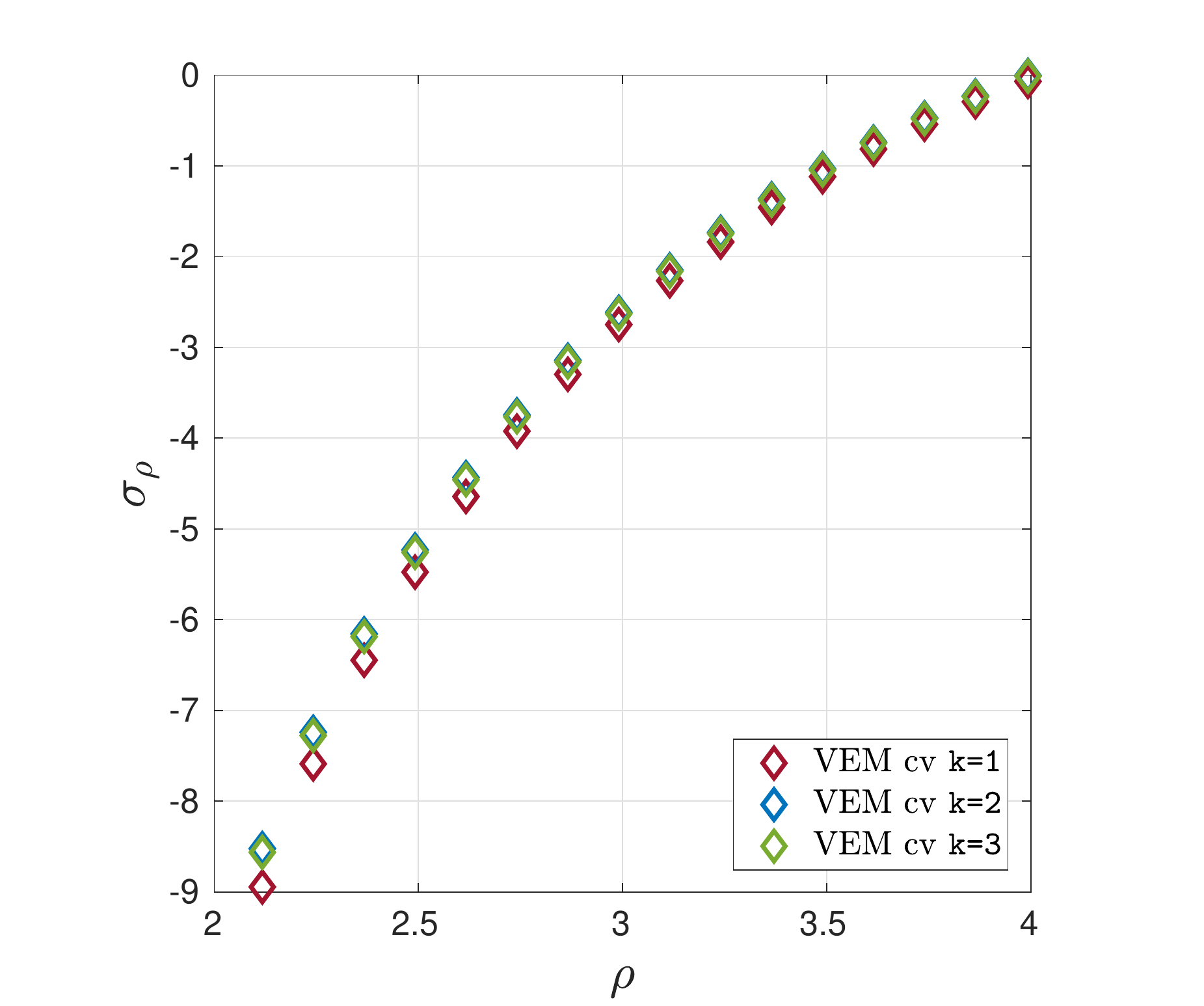} \\
(a) & (b)
\end{tabular}
\end{center}
\caption{Example 3. Thick-walled viscoelastic cylinder with internal pressure. (a) Location of radial quadrature nodes. (b) Radial plot for $\sigma_{\rho}$ stress component.}
\label{fig:sigma_rho}
\end{figure}

We finally compare again the accuracy of the novel VEM formulation with curved edges and the original straight edge formulation. In both cases we adopt quite coarse meshes: a mesh made by a single quadrilateral element (\texttt{quad1}), and meshes composed by $2\times 2$, $4 \times 4$ and $8 \times 8$ quadrilaterals (\texttt{quad4}, \texttt{quad16} and \texttt{quad64} respectively). It is emphasized that, in the curved case, the elements of the mesh have a curved boundary and thus no error is introduced in reproducing the geometry, while in the straight case the elements are straight quadrilaterals hence introducing a rectification error for the domain boundary approximation. 
The problem is the same described at the start of this section and the reference solution was obtained with FEAP as described above.

In Table~\ref{tab:dispAandB} we report the error in the radial displacement at points A and B (for the final time instant $t=20$), comparing, for various meshes and polynomial orders $k$, the straight case with the curved case. The advantage of the curved formulation (also considering that the computational cost is essentially equivalent to that of the straight formulation) appears clearly, in particular for higher values of $k$ and finer meshes.

%
\begin{table}[!htb]
\begin{center}
{\large Error on the radial displacement for point $A$}\\[1.5mm]
\begin{tabular}{|r|cc|cc|cc|}
\cline{2-7}
\multicolumn{1}{c}{}&
\multicolumn{2}{|c|}{$k=1$} &
\multicolumn{2}{c|}{$k=2$} &
\multicolumn{2}{c|}{$k=3$}\\
\cline{2-7}
\multicolumn{1}{c|}{}&VEM $s$ &VEM $cv$ &VEM $s$ &VEM $cv$ &VEM $s$ &VEM $cv$\\
\hline
\texttt{quad1}  &7.2993e-01  &7.0206e-01 &3.1914e-01 &1.4912e-01 &1.1079e-01 &7.3008e-02 \\
\texttt{quad4}  &5.0015e-01  &4.5682e-01 &5.6851e-02 &1.9920e-02 &4.4452e-02 &2.0313e-02 \\
\texttt{quad16} &2.1915e-01  &1.8054e-01 &9.3065e-03 &9.2705e-03 &1.5322e-02 &2.0427e-03 \\
\texttt{quad64} &6.7435e-02  &5.2200e-02 &2.6149e-03 &2.0149e-03 &4.2298e-03 &1.3574e-04 \\
\hline
\end{tabular} 
\end{center}
\begin{center}
{\large Error on the radial displacement for point $B$}\\[1.5mm]
\begin{tabular}{|r|cc|cc|cc|}
\cline{2-7}
\multicolumn{1}{c}{}&
\multicolumn{2}{|c|}{$k=1$} &
\multicolumn{2}{c|}{$k=2$} &
\multicolumn{2}{c|}{$k=3$}\\
\cline{2-7}
\multicolumn{1}{c|}{}&VEM $s$ &VEM $cv$ &VEM $s$ &VEM $cv$ &VEM $s$ &VEM $cv$\\
\hline
\texttt{quad1}  &6.2943e-01  &6.4095e-01  &3.6412e-01 &3.7846e-01 &3.0301e-01 &4.0254e-01 \\
\texttt{quad4}  &4.4898e-01  &4.2437e-01  &1.5582e-01 &1.1627e-01 &4.4705e-02 &2.3491e-02 \\
\texttt{quad16} &1.9818e-01  &1.6598e-01  &3.7664e-02 &2.3992e-02 &1.0348e-02 &1.0121e-04 \\
\texttt{quad64} &6.1086e-02  &4.7770e-02  &7.6054e-03 &3.7863e-03 &3.0215e-03 &2.4937e-05 \\
\hline
\end{tabular} 
\end{center}
\caption{Example 3. Thick-walled viscoelastic cylinder with internal pressure. Comparison of curved and straight VEMs. Relative errors on the radial displacement for points $A$ and $B$.}
\label{tab:dispAandB}
\end{table}
\paragraph{Example 4. Perforated plastic plate.}
\label{ss:plasticstrip}
A rectangular strip with width of $2L = 200$ $\textrm{mm}$ width and length of $2H = 360$ $\textrm{mm}$ with a circular hole of $2R = 100$ $\textrm{mm}$ diameter in its center is considered (see Fig. \ref{fig:strip_wt_hole}). Material response here follows classical von Mises plastic constitutive model, with material parameters: $E = 7000$ $\textrm{kg} / \textrm{mm}^2$, $\nu = 0.3$,  and yield stress $\sigma_{\textrm{y},0} = 24.3$ $\textrm{kg} / \textrm{mm}^2$ \cite{Zienkiewicz_Taylor_Zhu13}. Plane strain assumption is assumed, and a standard backward Euler scheme with return map projection is used for stress and material moduli computation at the quadrature point level \cite{simo_computational_1998}. Displacement boundary restraints are prescribed for normal components on symmetry boundaries and on top and lateral boundaries. Loading is applied by a uniform normal displacement $\delta = 2$ $\textrm{mm}$ with $400$ equal increments on the upper edge, see Fig.~\ref{fig:strip_wt_hole}. 

Owing to symmetry, only one quadrant of the perforated strip is discretized, as shown in Fig.~\ref{fig:strip_wt_hole_mesh}, with structured quadrilaterals, and a centroid based Voronoi tessellation, respectively. The simulation refers to virtual elements of order $k=2,3$, in the straight and curved variant, respectively.

Accuracy and robustness is assessed by plotting the structural response of the structure, in terms of force reaction sum at the imposed displacement top edge {\it vs.} integration step, which seems correct for all compared methods and mesh types, showing no significant spurious locking phenomena { (at least for the presented problem and polynomial degrees)}. The present benchmark confirms the proposed novel curved VEM methodology as a powerful tool to be implemented in a standard nonlinear FEM structural analysis platform.

\begin{figure}[!htb]
\centering
\includegraphics[bb=50 0 600 825, clip, angle=0, scale=0.3]{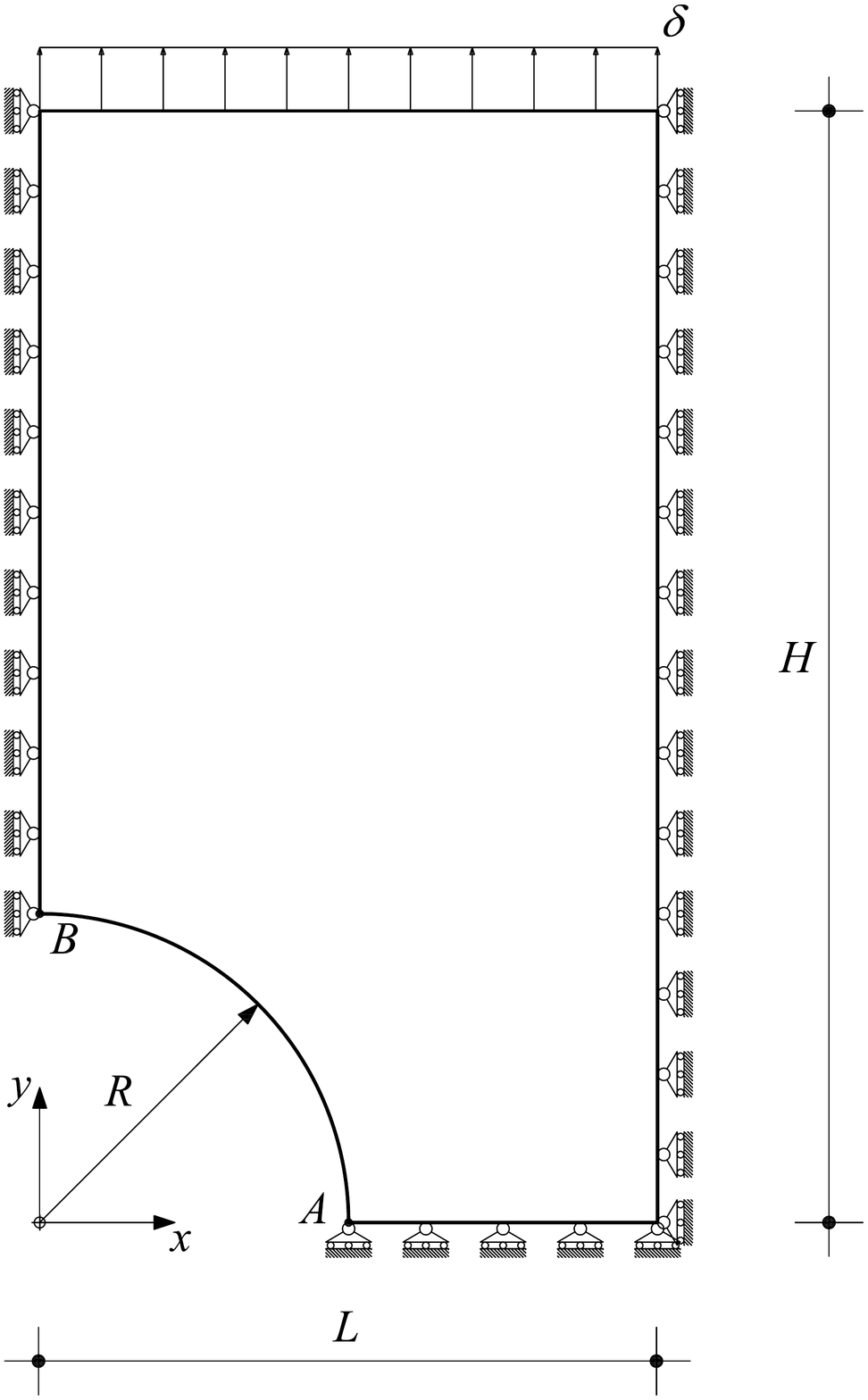}
\caption{Example 4. Perforated plastic plate. Geometry, boundary conditions, loading.}
\label{fig:strip_wt_hole}
\end{figure}

\begin{figure}[!htb]
\begin{center}
\begin{tabular}{ccc}
\includegraphics[width=0.26\textwidth]{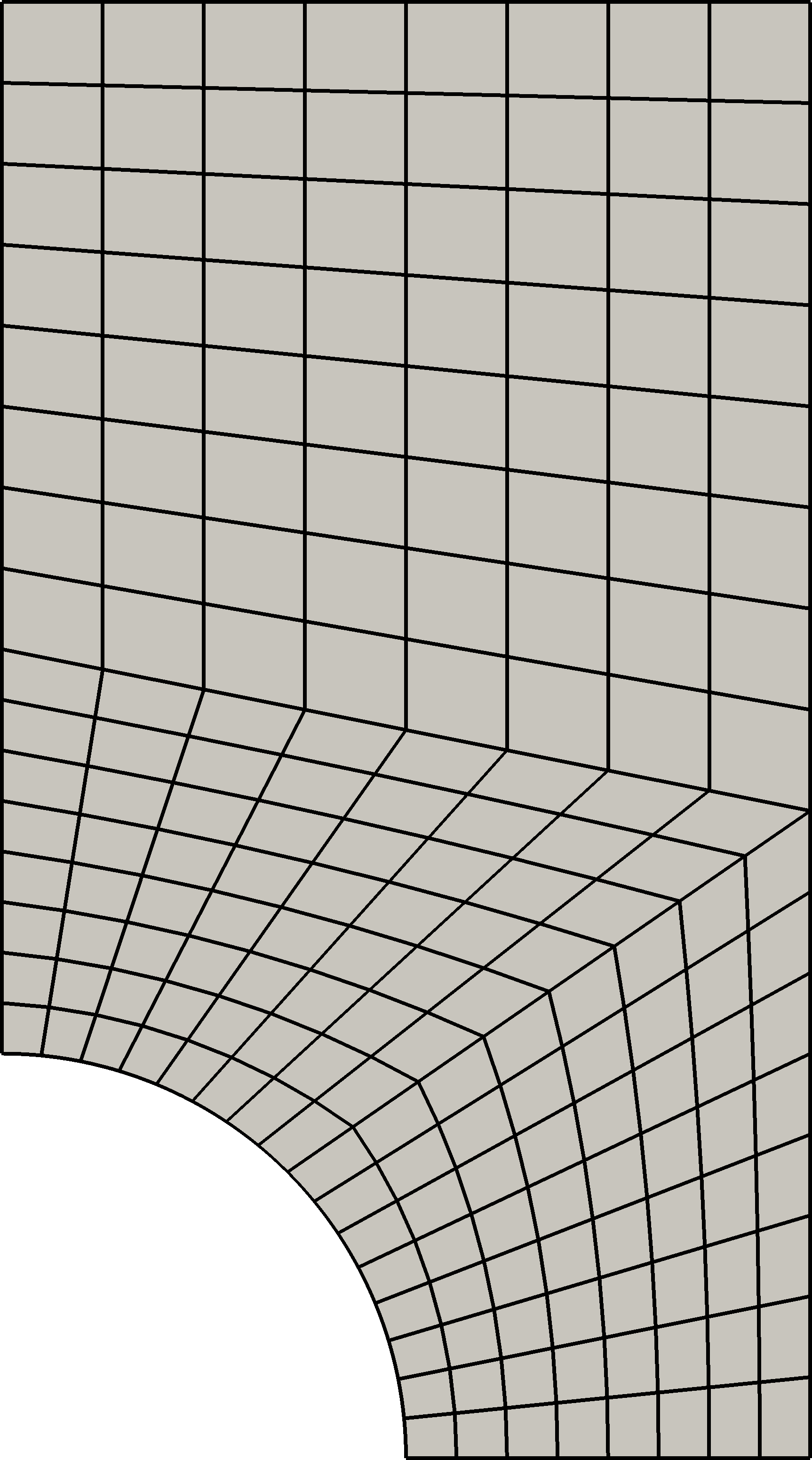}&\phantom{mmm}& 
\includegraphics[width=0.26\textwidth]{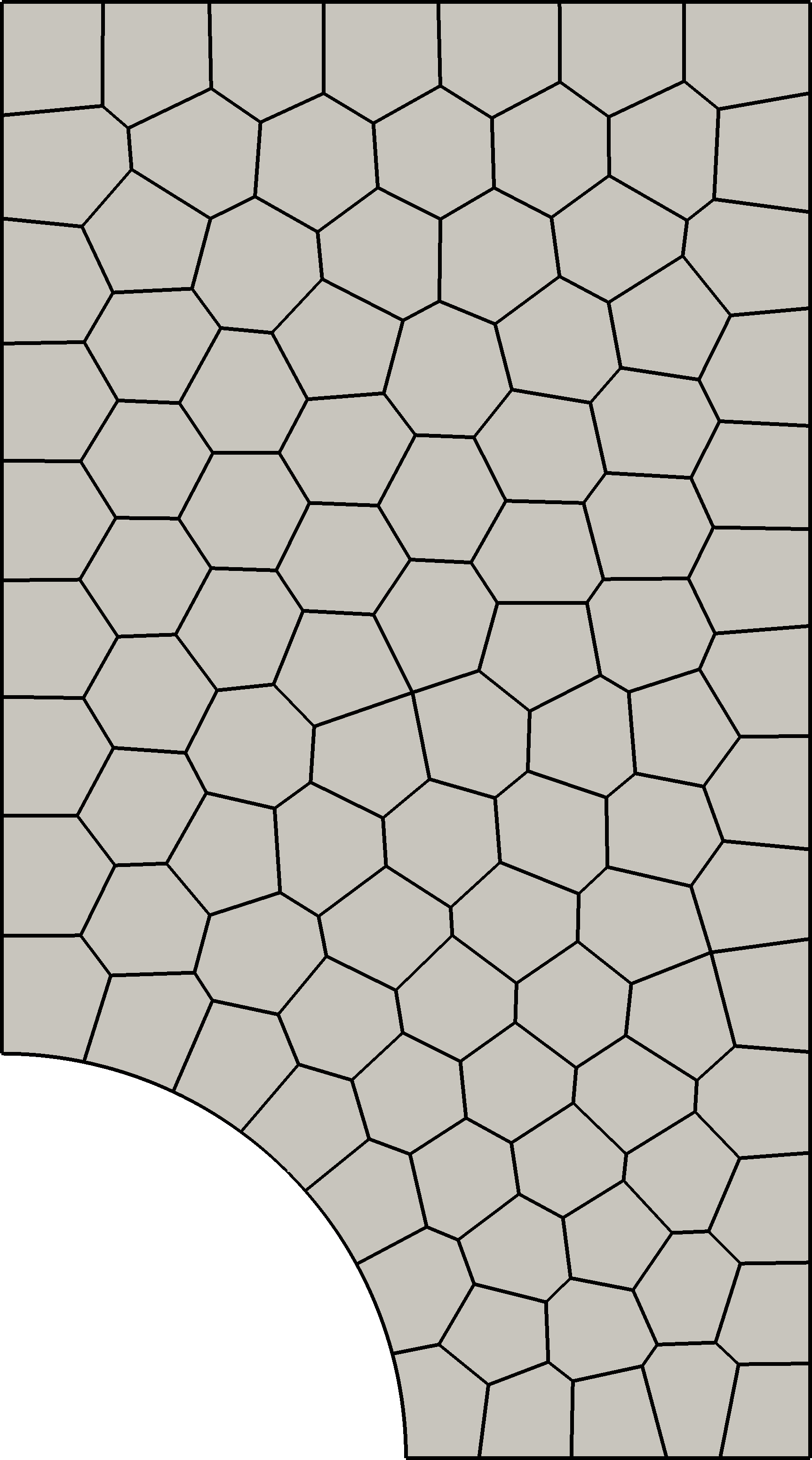}\\
&&\\
(a) & &(b)
\end{tabular} 
\end{center}
\caption{Example 4. Perforated plastic plate. (a) Structured quadrilateral mesh. (b) Voronoi mesh.}
\label{fig:strip_wt_hole_mesh}
\end{figure}

\begin{figure}
\begin{center}
\begin{tabular}{cc}
\includegraphics[width=0.48\textwidth]{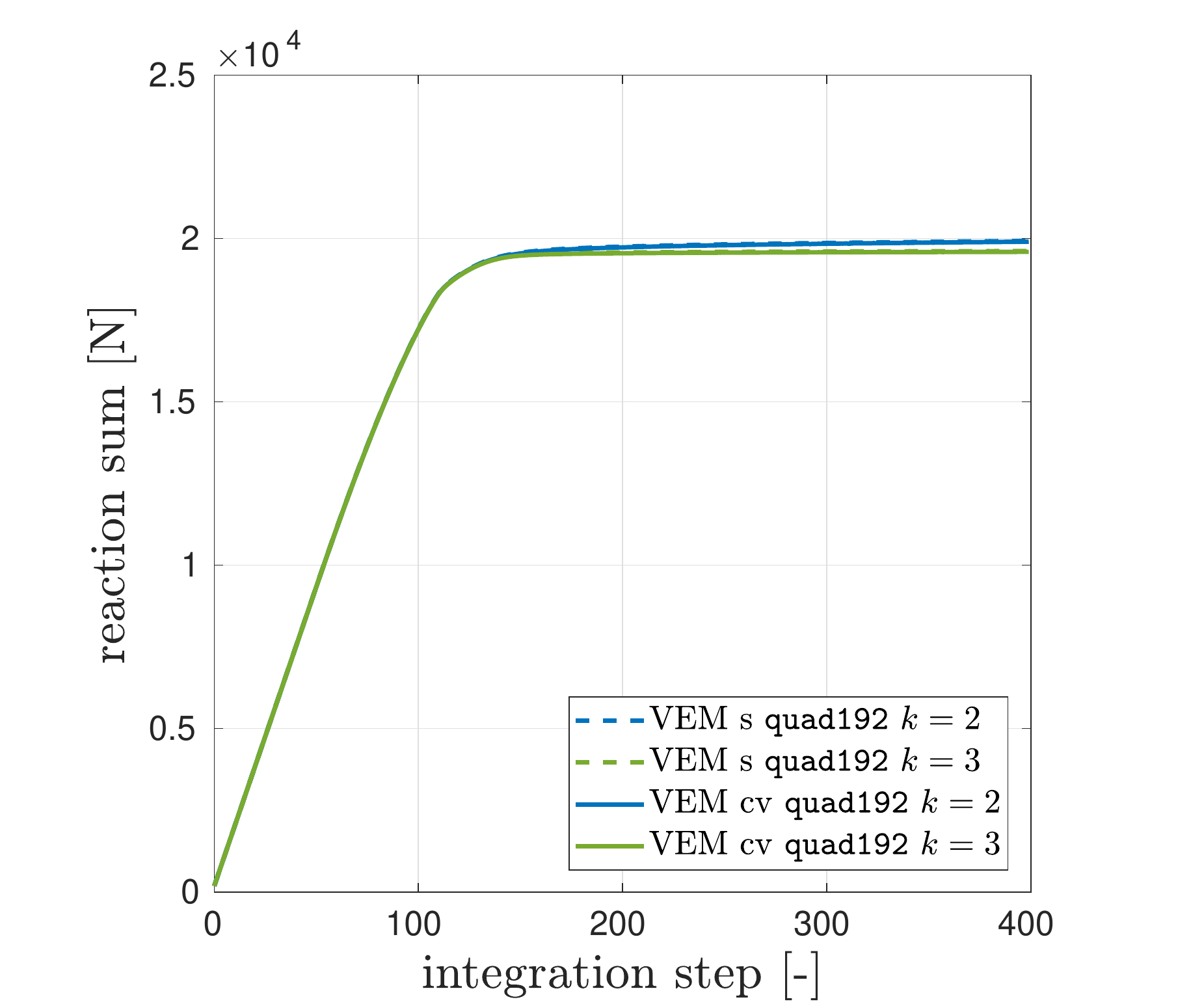} &
\includegraphics[width=0.48\textwidth]{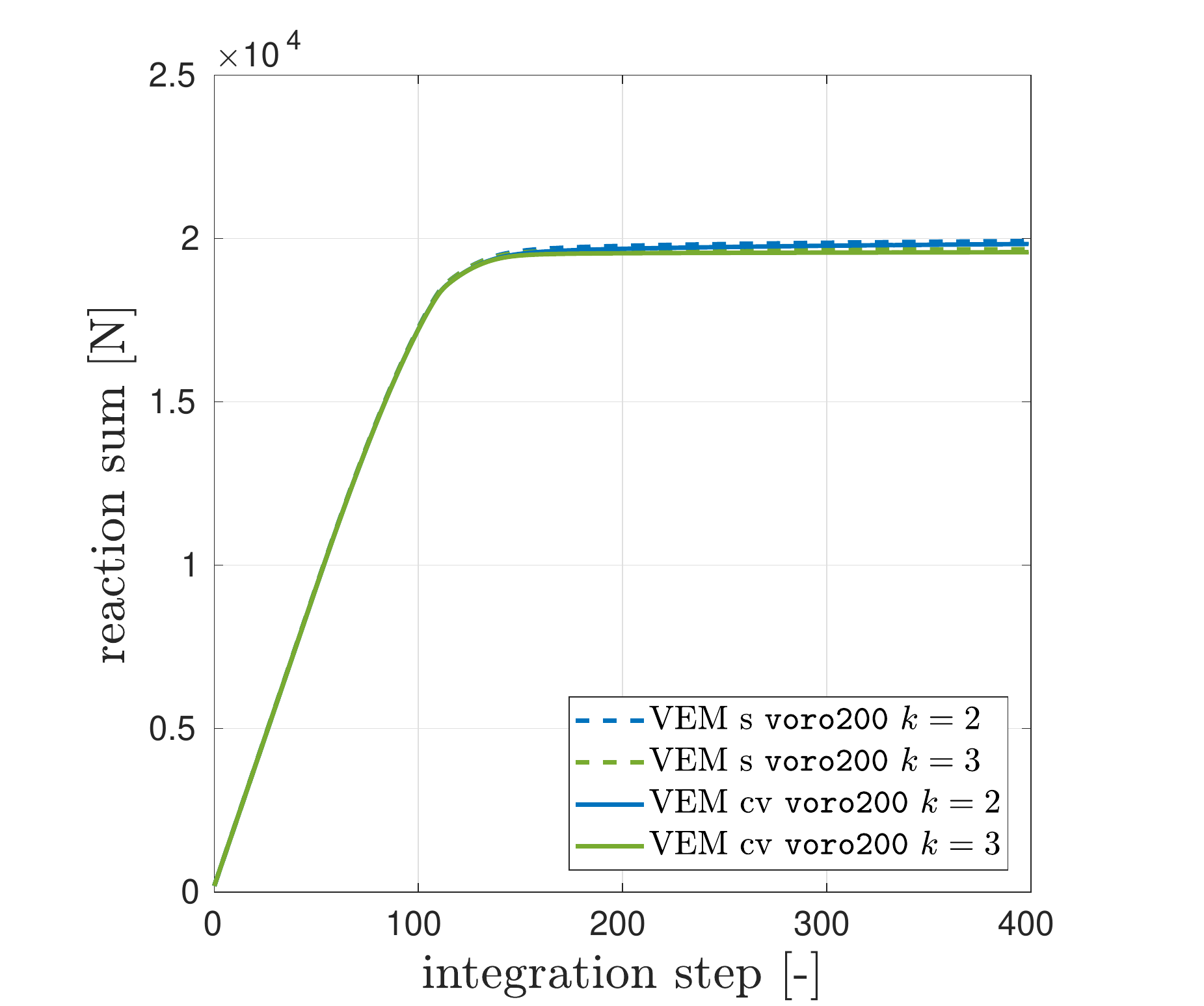}\\
(a) & (b)
\end{tabular}
\end{center}
\caption{Example 4. Perforated plastic plate. Structural response. (a) Quadrilateral mesh. (b) Voronoi mesh.}
\label{fig:strip_resp}
\end{figure}


\section{Conclusions}

We generalized the curvilinear Virtual Element technology to generic 2D solid mechanics problems in the small strain regime. 
The proposed approach can accept a generic black-box (elastic or inelastic) constitutive algorithm and, in addition, can make use of polygonal meshes with curved edges thus leading to an exact approximation of the geometry. 
We also introduced a novel Virtual Element space for displacements on curvilinear elements that, differently from the original one, contains also rigid body motions.
After presenting rigorous theoretical interpolation properties for the new space, we developed an extensive numerical test campaign to assess the behavior of the scheme. The campaign included a convergence analysis on a nonlinear elastic problem, a patch test to verify rigid body motions, standard benchmarks with inelastic materials, a study on the integration points rule to be adopted, and a comparison among curvilinear and standard Virtual Elements. It is noted that a curvilinear implementation does not imply any significant complication with respect to standard VEM technology and can hence be implemented into existing codes at very limited expense. Numerical results are very promising and indicate the proposed curvilinear VEM as a viable and competitive technology.

\section*{Acknowledgements}
E. Artioli gratefully acknowledges the partial financial support of the University of Rome Tor Vergata Mission Sustainability
Programme through project SPY-E81I18000540005; and the partial financial support of PRIN 2017 project ''3D PRINTING: A BRIDGE TO THE FUTURE (3DP\_Future). Computational methods, innovative applications, experimental validations of new materials and technologies.'', grant 2017L7X3CS\_004.

L. Beir\~ao da Veiga and F. Dassi were partially supported by the European Research Council through
the H2020 Consolidator Grant (grant no. 681162) CAVE, Challenges and Advancements in Virtual Elements. This support is gratefully acknowledged.

\bibliographystyle{plain}
\bibliography{VEM,VEM2}

\end{document}